\documentclass{amsart}

\usepackage{amsmath, amsthm, amsfonts,stmaryrd,amssymb}
\usepackage[top=1in, bottom=1.25in, left=1.25in, right=1.25in]{geometry}
\usepackage{mathrsfs}    
\usepackage{ dsfont }                  
\usepackage{mathtools}                   
\usepackage{booktabs}
\usepackage{tikz}
\usepackage{todonotes}
\usepackage[all]{xy}
\usepackage[utf8]{inputenc}
\usepackage{mathtools}                     
\usepackage{todonotes}
\usepackage{faktor}
\usepackage{ marvosym }
\usepackage{bbm}                         
\usepackage{framed}
\usepackage{graphicx}   

\usepackage{xcolor}
\newif\ifcomments
\commentstrue
\newcommand\comment[1]{%
  \ifcomments
  \textcolor{black}{#1}
  \else
  \fi}
\definecolor{mycolor}{RGB}{20, 20, 122}

\usepackage{theoremref}
\numberwithin{equation}{section}
\theoremstyle{plain}
\newtheorem{remark}{Remark}
\newtheorem{theorem}{Theorem}
\newtheorem{theorem*}{Theorem \nonumber}

\newtheorem{corollary}[theorem]{Corollary}
\newtheorem{proposition}[theorem]{Proposition}
\newtheorem{example}[theorem]{Example}

\theoremstyle{definition}
\newtheorem{definition}{Definition}

\def\th{\theta} 
 
\def\Two{\mathrm{II}}

 \def\m{\mathsf{m}}

\def\x{\times}

\def\T{\mathbb{T}}

\def\R{{\mathbb R}}
\def\C{{\mathbb C}}

\def\injs{ Inj}

\def\exp{\operatorname{exp}}

\def\Id{\operatorname{Id}}
\def\Dens{\operatorname{Dens}}
\def\Div{\operatorname{div}}

\def\Diff{\operatorname{Diff}}

\let\on=\operatorname

\newcommand{\ud}{\,\mathrm{d}}

\newcommand{\eqdef}{\ensuremath{\stackrel{\mbox{\upshape\tiny def.}}{=}}}

\begin{document}
\title[The Camassa-Holm equation as an incompressible Euler equation]{The Camassa-Holm equation as an incompressible Euler equation: a geometric point of view}
\author{Thomas Gallou\"et} 
\address{CMLS, UMR 7640, \'Ecole Polytechnique, FR-91128 Palaiseau Cedex.}
\email{thomas.gallouet@polytechnique.edu}
\author{Fran\c{c}ois-Xavier Vialard}
\address{Universit\'e Paris-Dauphine, PSL Research University, Ceremade \\ INRIA, Project team Mokaplan}
\email{fxvialard@normalesup.org}
\maketitle

\begin{abstract}
The group of diffeomorphisms of a compact manifold endowed with the $L^2$ metric acting on the space of probability densities gives a unifying framework for the incompressible Euler equation and the theory of optimal mass transport. Recently, several authors have extended optimal transport to the space of positive Radon measures where the Wasserstein-Fisher-Rao distance is a natural extension of the classical $L^2$-Wasserstein distance. In this paper, we show a similar relation between this unbalanced optimal transport problem and the $H^{\Div}$ right-invariant metric on the group of diffeomorphisms, which corresponds to the Camassa-Holm (CH) equation in one dimension. 
Geometrically, we present an isometric embedding of the group of diffeomorphisms endowed with this right-invariant metric in the automorphisms group of the fiber bundle of half densities endowed with an $L^2$ type of cone metric. 
This leads to a new formulation of the (generalized) CH equation as a geodesic equation on an isotropy subgroup of this automorphisms group; On $S_1$, solutions to the standard CH thus give radially $1$-homogeneous solutions of the incompressible Euler equation on $\R^2$ which preserves a radial density that has a singularity at $0$. 
An other application consists in proving that smooth solutions of the Euler-Arnold equation for the $H^{\Div}$ right-invariant metric are length minimizing geodesics for sufficiently short times. 
\end{abstract}


\section{Introduction}
In his seminal article \cite{Arnold1966}, Arnold showed that the incompressible Euler equation can be viewed as a geodesic flow on the group of volume preserving diffeomorphisms of a Riemannian manifold $M$. His formulation had an important impact in the mathematical literature and it has led to many different works. Among others, let us emphasize two different points of view which have proven to be successful. 
\par The first one has been investigated by Ebin and Marsden in \cite{Ebin1970} where the authors have taken an intrinsic point of view on the group of diffeomorphisms as an infinite dimensional weak Riemannian manifold. Formulating the geodesic equation as an ordinary differential equation in a Hilbert manifold of Sobolev diffeomorphisms, they proved, among others, local well-posedness of the geodesic equation for smooth enough initial conditions. Since then, many fluid dynamic equations, including the Camassa-Holm equation, have been written as a geodesic flow on a group of diffeomorphisms endowed with a right-invariant metric or connection \cite{Kouranbaeva1999,KhesinCurvature,Misiolek2002,Escher2011,Holm1998} and analytical properties have been derived in the spirit of \cite{Ebin1970}. Note in particular that all these works assume a strong ambient topology such as $H^s$ for $s$ high enough and the topology given by the Riemannian metric is generically weaker, typically $L^2$ in the case of incompressible Euler.
\par Another point of view, motivated by the variational interpretation of geodesics as minimizers of the action functional, was initiated by Brenier. He developed an extrinsic approach by considering the group of volume preserving diffeomorphisms as a Riemannian submanifold embedded in the space of maps $L^2(M,M)$ which is particularly simple when $M$ is the Euclidean space or torus. In particular, his polar factorization theorem \cite{Brenier1991} was motivated by a numerical scheme approximating geodesics on the group of volume preserving diffeomorphisms. Optimal transport then appeared as a key tool to project a map onto this group by minimizing the $L^2$ distance and it can be interpreted as a non-linear extension of the pressure in the incompressible Euler equation. Since then, optimal transport has witnessed an impressive development and found many important applications inside and outside mathematics, see for instance the gigantic monograph of Villani \cite{villani2008optimal}. Brenier also used optimal transport in order to define the notion of generalized geodesics for the incompressible Euler equation in \cite{Brenier1999}.
\par
In this article, we develop Brenier's point of view for a generalization in any dimension of the Camassa-Holm equation. Indeed, we present an isometric embedding of the group of diffeomorphisms endowed with the right-invariant $H^{\Div}$ metric into a space of maps endowed with an $L^2$ metric. Moreover, the recently introduced Wasserstein-Fisher-Rao distance \cite{GeneralizedOT1,GeneralizedOT2}, a generalization of optimal transport to measures that do not have the same total mass, plays the role of the $L^2$ Wasserstein distance for the incompressible Euler equation. 
\subsection{Contributions}

The underlying key point for our work is the generalization of the (formal) Riemannian submersion already presented in \cite{GeneralizedOT2}, which unifies the unbalanced optimal problem and the $H^{\Div}$ right-invariant metric.
We rewrite the geodesic flow of the right-invariant $H^{\Div}$ metric on the diffeomorphism group as a geodesic equation on a constrained submanifold of a semidirect product of group or equivalently on the automorphism group of the half-densities fibre bundle endowed with the cone metric (see Section \ref{Sec:ConeMetric} for its definition). This point of view has three applications:  (1) We interpret solutions to the Camassa-Holm equation and one of its generalization in higher dimension as particular solutions of the incompressible Euler equation on the plane for a radial density which has a singularity at $0$. This correspondence can be introduced via a sort of Madelung transform. 
(2) We generalize a result of Khesin et al. in \cite{KhesinCurvature} by computing the curvature of the group as a Riemannian submanifold. 
(3) Generalizing a result of Brenier to the case of Riemannian manifolds, which states that solutions of the incompressible Euler equation are length minimizing geodesic for sufficiently short times, we prove similar results for the Camassa-Holm equation.
\par
Since the interpretation of the Camassa-Holm equation as an incompressible Euler equation is one of the main results of the paper, we present it below.
\begin{theorem*}[Camassa-Holm as incompressible Euler]
Solutions to the Camassa-Holm equation on $S_1$
\begin{equation}
\partial_{t} u - \frac14 \partial_{txx} u +3 \partial_{x} u \,u  - \frac12\partial_{xx} u \,\partial_x u - \frac14\partial_{xxx} u \,u = 0\,
\end{equation}
are mapped to solutions of the incompressible Euler equation on $\R^2 \setminus \{ 0\}$ for the density $ \rho = \frac{1}{r^4} \on{Leb}$, that is
\begin{equation}\label{Eq:EulerSimple}
\begin{cases}
\dot{v} + \nabla_v v = -\nabla P\,,\\
\nabla \cdot (\rho v) = 0\,,
\end{cases}
\end{equation}
by the map $u \mapsto \left(u(\theta),\frac{r}{2} \partial_x u(\theta)\right)$.
\end{theorem*}
In other words, rewriting the Camassa-Holm equation in polar coordinates transforms it into an incompressible Euler equation. Obviously, the proof of the theorem can be reduced to a simple calculation. In this paper, we show the geometrical structures that underpin this formulation.

\subsection{Link to previous works}
Recently, several authors including the second author extended optimal transport to the case of unbalanced measures, i.e. measures that do not have the same total mass. Although several works extended optimal transport to this setting, surprisingly enough, the equivalent of the $L^2$-Wasserstein distance in this unbalanced setting has been introduced in 2015 simultaneously by \cite{GeneralizedOT1,GeneralizedOT2} motivated by imaging applications, \cite{LieroMielkeSavareLong,LieroMielkeSavareShort} motivated by gradient flows as well as \cite{new2015kondratyev} and by \cite{Rezankhanlou2015} for optimal transport of contact structures.
In this paper, we show that, in the case of the Wasserstein-Fisher-Rao metric, the equivalent to the incompressible Euler equation is a generalization of the Camassa-Holm equation, namely the Euler-Arnold equation for the right-invariant metric $H^{\Div}$ on the group of diffeomorphisms. In one dimension, geodesics for the right-invariant $H^{\Div}$ metric are the solutions to the Camassa-Holm equation introduced in \cite{CH}. Since its introduction, the Camassa-Holm equation has attracted a lot of attention  since it is a bi-Hamiltonian system as well as an integrable system, it exhibits peakon solutions and it is a model for waves in shallow water \cite{Constantin2008,Constantin2001,Lenells2005,ConstantinEscher1998,BressanConstantin2007,danchin2001,GRUNERT2011}. In particular, this equation is known for its well understood blow-up in finite time and is a model for wave breaking \cite{McKean2004}. 
\par 
Although the title of \cite{Bressan2005}, which refers to optimal transport and the Camassa-Holm equation, is seemingly close to our article, the authors introduce a metric based on optimal transport which gives Lipschitz estimates for the solutions of the Camassa-Holm equation and it is a priori completely different to our construction. Indeed, in our article, the optimal transport metric measures the discrepancy of not being in the stabilizer of the group action defined in Section \ref{Sec:Automorphisms} where the solutions of the Camassa-Holm equation lie. 

Maybe more related to our results, homogeneous solutions of Euler equations have been studied for example in \cite{2017Elgindi,LuoShvydkoy}, however the measure preserved in those works is not a singular measure, as in our work.

\subsection{Plan of the paper}
In Section \ref{Sec:GeomViewPoint}, we recall the link between optimal transport and the incompressible Euler equation, then we introduce the Wasserstein-Fisher-Rao metric which generalizes the $L^2$ Wasserstein metric on the space of \emph{probability} densities to the space of \emph{integrable} densities, thus relaxing the mass constraint. 
We present the generalization of Otto's Riemannian submersion to this unbalanced case. This generalization uses a semidirect product of group which can be interestingly interpreted as the automorphism group of the principal fibre bundle of half-densities, as explained in Section \ref{Sec:Automorphisms}. This semidirect product of group has a natural left action on the space of densities and it gives the Riemannian submersion between an $L^2$ type of metric on the group and the Wasserstein-Fisher-Rao metric on the space of densities.
\par
In Section \ref{Sec:EulerArnold}, we briefly review the result on the local well-posedness of the Camassa-Holm equation and its $H^{\Div}$ generalization and the associated metric properties.
\par
Section \ref{Sec:Submanifold} presents the corresponding submanifold point of view corresponding to the Camassa-Holm equation (its generalization). The submanifold is the isotropy subgroup of the left action of the semidirect product of group and the ambient metric is the $L^2$ type of metric. As a direct consequence, it gives a generalization of a result on the sectional curvature written in \cite[Theorem A.2]{KhesinCurvature}.
\par
The two main applications of our approach are detailed in Section \ref{Sec:Applications}. The one dimensional case is developed in section \ref{Sec:CHAsEuler} where we show that solutions of the Camassa-Holm equation (its generalization) can be seen as particular solutions of an incompressible Euler equation for a particular density on the cone which has a singularity at $0$. We improve a result of Ebin and Marsden in dimension $1$ by extending Brenier's approach to show that every smooth geodesics are length minimizing on a sufficiently short time interval under mild conditions. Then, these result are generalized in \ref{Sec:GeneralHdiv}.

\subsection{Notations}
Hereafter is a non exhaustive list of notations used throughout the paper.
\begin{itemize}
\item $(M,g)$ is a smooth orientable Riemannian manifold which is assumed compact and without boundary. Its volume form is denoted by $\on{vol}$, $TM$ and $T^*M$ denote respectively the tangent and the cotangent bundle.
\item The distance on $(M,g)$ is sometimes denoted by $d_M$ when a confusion might occur.
\item For $x\in M$, the squared norm of a vector $v \in T_xM$ will be denoted by $\| v \|^2$ or $g(x)(v,v)$.
\item For $x \in M$, we denote by $\exp^M_x: T_xM \to M$, the exponential map, the superscript being  a reminder of the underlying manifold.
\item $\mathcal{C}(M)$ is the Riemannian cone over $(M,g)$ and is introduced in Definition \ref{def:Cone}.
\item The operator $\on{div}$ is the divergence w.r.t. the volume form on $(M,g)$.
\item The Lie bracket between two vector fields $X,Y$ on $M$ is denoted by $[X,Y]$.
\item If $f \in C^1(M,\R)$, then $\nabla f$ is the gradient of $f$ w.r.t. the metric $g$. Sometimes, we use the notation $\nabla_x$ to make clear which variable we consider.
\item The group of invertible linear maps on $\R^d$ is denoted by $\on{GL}_d(\R)$.
\item For a quantity $f(t,x)$ that depends on time and space variable, we denote by $\dot{f}$ its time derivative.
\item On $\R$ and $\C$, $| \cdot |$ denotes respectively the absolute value and the module.
\item $M = S_n(r)$ the Euclidean sphere of radius $r$ in $\R^{n+1}$. 
\item The Lebesgue measure is denoted by $\on{Leb}$.
\item Sometimes, we use the notation $a\eqdef b$ to define $a$ as $b$.
\end{itemize}


\section{A Geometric Point of View on Unbalanced Optimal Transport}\label{Sec:GeomViewPoint}

Before presenting unbalanced optimal transport in more details, we give a brief overview of the link between optimal transport and the incompressible Euler equation.

\subsection{Optimal transport and the incompressible Euler equation}
 We first start from the usual static formulation of optimal transport and then present the dynamical formulation proposed by Benamou and Brenier. The link between the two formulations can be introduced via Otto's Riemannian submersion, which also provides a clear connection between incompressible Euler equation and the dynamical formulation of optimal transport. Our presentation closely follows the discussion in \cite[Appendix A.5]{khesin2008geometry} and interesting complements can be found in 
\cite{Modin2012,KhesinCurvature,Khesin2013}. In the rest of the section, unless otherwise mentioned, $M$ denotes a smooth Riemannian manifold without boundary, for instance the flat torus.
\par
\textbf{Static formulation of optimal mass transport: }The optimal mass transport problem as introduced by Monge in 1781 consists in finding, between two given probability measures $\nu_1$ and $\nu_2$, a map $\varphi$ such that $\varphi_* \nu_1 = \nu_2$, i.e. the image measure of $\nu_1$ by $\varphi$ is equal to $\nu_2$ and which minimizes a cost given by
\begin{equation} \label{Eq:MongeFormulation0}
\int_M c(x,\varphi(x)) \ud \nu_1(x) \,, 
\end{equation}
where $c$ is a positive function that represents the cost of moving a particule of unit mass from location $x$ to location $y$. This problem is ill-posed in the sense that solutions may not exist and the Kantorovich formulation of the problem is the correct relaxation of the Monge formulation, which can be presented as follows: On the space of probability measures on the product space $M\times M$, denoted by $\mathcal{P}(M \times M)$, find a minimizer to 
\begin{equation}
\mathcal{I}(m) = \int_{M^2} c(x,y) \ud m(x,y) \, \text{ such that } p^1_*(m) = \nu_1 \text{ and } p^2_*(m) = \nu_2 \,,
\end{equation}
where $p^1_*(m),p^2_*(m)$ denote respectively the image measure of $m\in \mathcal{P}(M\times M)$ under the projections on the first and second factors on $M\times M$. Most often in the litterature, the cost $c$ is  chosen as a power of a distance. From now on, we will only discuss the case $c(x,y) = d(x,y)^2$ where $d$ is the distance associated with a Riemannian metric on $M$. In this case, the Kantorovich minimization problem defines the so-called $L^2$-Wasserstein distance on the space of probability measures. The Monge formulation can be expressed as a minimization problem as follows
\begin{equation}
W_2(\mu, \nu)^2 \eqdef \inf_{\varphi \in \Diff(M)} \left\{ \int_M d(\varphi(x), x)^2 \, \ud \nu_1(x) \, : \, \varphi_*\nu_1 = \nu_2 \right\}\,,
\end{equation}
where $\Diff(M)$ denotes the group of smooth diffeomorphisms of $M$. 
\par 
\textbf{Dynamic formulation: }
In \cite{benamou2000computational}, Benamou and Brenier introduced a dynamical version of optimal transport which was inspired and motivated by the study of the incompressible Euler equation. Let $\rho_0,\rho_1 \in C^\infty(M,\R_+)$ be integrable densities, note that all the quantities will be implicitly time dependent. 
The dynamic formulation of the Wasserstein distance consists in minimizing
\begin{equation}\label{BB}
\mathcal{E}(v) = \int_0^1 \int_M \|v(t,x)\|^2\rho(t,x) \,\ud \!\on{vol}(x) \ud t \ \, ,
\end{equation}
subject to the constraints $\dot{\rho} + \Div(v \rho) = 0$ and initial condition $\rho(0) = \rho_0$ and final condition $\rho(1) = \rho_1$. The notation $\| \cdot \|$ stands for the Euclidean norm. \par
Equivalently, following \cite{benamou2000computational}, a convex reformulation using the momentum $\m = \rho v$ reads
\begin{equation} \label{ConvexReformulation}
\mathcal{E}(\m) =   \int_0^1 \int_M \frac {\|\m(t,x) \|^2}{\rho(t,x)} \,\ud \!\on{vol}(x)  \ud t\, ,
\end{equation}
subject to the constraints $\dot{\rho} + \Div(\m) = 0$ and initial condition $\rho(0) = \rho_0$ and final condition $\rho(1) = \rho_1$. Let us underline that the functional  $\mathcal{E}$ is convex in $\rho,\m$ and the continuity equation is linear in $(\rho,\m)$, therefore convex optimization methods can be applied for numerical purposes. Due to the continuity equation, the problem is feasible if and only if the initial and final densities have the same total mass using Moser's lemma \cite{MoserLemma}.

\textbf{Otto's Riemannian submersion: }
The link between the static and dynamic formulations is made clear using Otto's Riemannian submersion \cite{OttoPorousMedium} which emphasizes the idea of a group action on the space of probability densities.
 Let $\Dens_p(M)$ be the set of probability measures that have smooth positive densities with respect to the volume measure $\on{vol}$. We consider such a probability density denoted by $\rho_0$. Otto showed that the map 
\begin{align*}
& \pi: \Diff(M) \to \Dens_p(M) \\
&\pi(\varphi) = \varphi_* (\rho_0)
\end{align*} 
is a formal Riemannian submersion of the metric $L^2(\rho_0)$ on  $ \Diff(M)$ to the $L^2$-Wasserstein metric on $\Dens_p(M)$. For all the basic properties of Riemannian submersions, we refer the reader to  \cite{GallotHulinLafontaine}. The fiber of this Riemannian submersion at point $\rho_0 \equiv 1$ is the subgroup of diffeomorphisms preserving the volume measure $\on{vol}$, we denote it by $\on{SDiff}(M)$ and we denote its tangent space at $\on{Id}$ by $\on{SVect}(M)$, the space of divergence free vector fields. The vertical space at a diffeomorphism $\varphi \in \Diff(M)$ for $\rho \eqdef \varphi_*\rho_0$ is  
\begin{equation}
\on{Vert}_\varphi = \left\{  v \circ \varphi \, ; \, v \in  \on{Vect}(M) \text{ s.t. } \Div (\rho v) = 0 \right\}\,.
\end{equation}
In particular, consider $\varphi \in \on{SDiff}(M)$, the vertical space is $\on{Vert}_\varphi = \left\{  v \circ \varphi \, ; \, v \in  \on{SVect}(M) \right\}$
and the horizontal space is
\begin{equation}
\on{Hor}_\varphi = \left\{  \nabla p \circ \varphi \, ; \, p \in C^\infty(M,\R) \right\}\,.
\end{equation}
\par 
\textbf{Incompressible Euler equation: } On the fiber $\on{SDiff}(M)$, the $L^2(\on{vol})$ metric is right-invariant. In Arnold's seminal work \cite{Arnold1966}, it is shown that the Euler-Lagrange equation associated with this metric is the incompressible Euler equation. Arnold derived this equation as a particular case of geodesic equations on a Lie group endowed with a right-invariant metric. In its Eulerian formulation, the incompressible Euler equation is, when $M = \T^d$ the flat torus for the Lebesgue measure,
\begin{equation} 
\begin{cases}
\partial_t v(t,x)  + v(t,x) \cdot \nabla v(t,x)=-\nabla p(t,x), \quad t>0,\;x\in  M  \, ,  \medskip\\
\on{div}(v) = 0\,,\\
v(0,x) = v_{0}(x)\,,
\end{cases} \label{Eq:EI}
\end{equation}
where $v_0 \in \on{SVect}(M)$ is the initial condition and $p$ is the pressure function. On a general Riemannian manifold $(M,g)$ compact and without boundary, the formulation is similar, omitting the time and space variables, for the volume measure,
\begin{equation}
\begin{cases}
\partial_t v  +  \nabla_{v} v=-\nabla p, \quad t>0,\;x\in  M  \, ,  \medskip\\
\on{div}(v) = 0\,,\\
v(0,x) = v_{0}(x)\,,
\end{cases} \label{eq:EIR}
\end{equation}
where, in this case, the symbol $\nabla$ denotes the Levi-Civita connection associated with the Riemannian metric on $M$ and $\on{div}$ denotes the divergence w.r.t. the volume measure.
Another fruitful point of view consists in considering the group $\on{SDiff}(M)$ as isometrically embedded in the group $\Diff(M)$ endowed with the $L^2(\on{vol})$ (non right-invariant) metric. Therefore the geodesic equations are simply geodesic equations on the Riemannian submanifold $\on{SDiff}(M)$ and the geodesic equations can be written as 
\begin{equation}\label{Eq:EulerLagrangian}
\ddot{\phi} = -\nabla p \circ \phi\,,
\end{equation}
where $\phi \in \on{SDiff}(M)$ and $p$ is still a pressure function. Using this Riemannian submanifold approach, Brenier was able to prove that solutions for which the Hessian of $p$ is bounded in $L^\infty$ are length minimizing for short times and several of his analytical results were derived from this formulation \cite{Brenier1993,Brenier1999}.
\par 
\textbf{Inviscid Burgers equation: } 
The geodesic equation on the group of diffeomorphisms for the $L^2$ metric written in Eulerian coordinates is the compressible Burgers equation. Its formulation on $M = \T^d$ is
\begin{equation}
\partial_t u(t,x) + u(t,x) \cdot \nabla u(t,x) = 0\,, 
\end{equation}
or on a general Riemannian manifold
\begin{equation}
\partial_t u + \nabla_{u}u = 0\,.
\end{equation}
This formulation is obviously related to the incompressible Euler equation where the pressure $p$ can be interpreted as a Lagrange multiplier associated with the incompressibility constraint, which is not present in Burgers equation. Since the map $\pi$ is a Riemannian submersion, geodesics on the space of densities can be lifted horizontally to geodesics on the group. These horizontal geodesics are potential solutions of the Burgers equation, if $u_0 = \nabla q_0$, i.e. $u$ is a potential at the initial time, then $u_t$ stays potential for all time (until it is not well defined any longer). The corresponding equation for the potential $q$ is the Hamilton-Jacobi equation
\begin{equation}\label{EqSimpleHJEquation}
\partial_t q(t,x) + \frac 12 \|\nabla q(t,x)\|^2 = 0\,,
\end{equation}
which, in this formulation, makes sense on a Riemannian manifold.

\subsection{The Wasserstein-Fisher-Rao metric,}
\par
\textbf{its dynamical formulation.}
The continuity equation enforces the mass conservation property in the Benamou-Brenier formulation \eqref{BB} (or \eqref{ConvexReformulation} recalling that by definition $\m = \rho v$). This constraint can be relaxed by introducing a source term $\mu$ in the continuity equation,
\begin{equation}\label{GCC}
\dot{\rho} =- \Div(\rho v)  + \mu\, =- \Div(\m)  + \mu\, .
\end{equation}
For a given variation of the density $\dot{\rho}$, there exist a priori many couples $(v,\mu)$ that reproduce this variation.
Following \cite{Metamorphosis2005}, it can be determined via the minimization of the norm of $(v,\mu)$, for a given choice of norm. The penalization of $\mu$ was chosen in \cite{OTmaasrumpf} as the $L^2$ norm but a natural choice is rather the Fisher-Rao metric $$ \on{FR}^2(\mu) = \int_M \frac {\mu(t,x)^2}{\rho(t,x)} \,\ud \!\on{vol}(x)  \,,$$
because it is homogeneous.
In other words, this is the $L^2$ norm of the growth rate w.r.t. the density $\rho$ since it can be written as $\int_M \alpha(t,x)^2\rho(t,x) \,\ud \!\on{vol}(x)$ where $\alpha$ is the growth rate $\alpha(t,x) \eqdef \frac{\mu(t,x)}{\rho(t,x)}$. Note in particular that this action is  $1$-homogeneous with respect to the couple $(\mu, \rho)$. This point is important for convex analysis properties and especially, in order to define the action functional on singular measures via the same formula. Obviously, there are many other choices of norms that satisfies this homogeneity property but this particular one can be related to the Camassa-Holm equation.
\par 
Thus, the Wasserstein-Fisher-Rao metric tensor denoted by $\on{WF}_{\rho}$ is simply given by the infimal convolution, a standard tool in convex analysis, between the Wasserstein and the Fisher-Rao metric tensors. Indeed, the metric tensor at a density $\rho$ is defined via the minimization
\begin{equation}
\on{WF}_{\rho}(\dot{\rho},\dot{\rho}) = \inf_{v,\alpha} \int_M \alpha(x)^2 + \|v(x)\|^2\,\ud \rho(x)\, \text{ s.t. } \dot{\rho} = -\Div(\rho v) + 2 \alpha \rho\,.
\end{equation}

 The distance associated with this metric tensor has been named Wasserstein-Fisher-Rao \cite{GeneralizedOT1}, Hellinger-Kantorovich \cite{LieroMielkeSavareLong}, Kantorovich-Fisher-Rao \cite{JKOKFR}. 
\begin{definition}[$\on{WF}$ metric] Let $(M,g)$ be a smooth Riemannian manifold compact and without boundary, $a,b \in \R_+^*$ be two positive real numbers and $\rho_0,\rho_1 \in \mathcal{M}_+(M)$ be two nonnegative Radon measures. The Wasserstein-Fisher-Rao metric is defined by
\begin{equation} \label{Extension1}
\on{WF}^2(\rho_0,\rho_1) = \inf_{\rho,\m,\mu} \mathcal{J}(\rho,\m,\mu)  \, ,
\end{equation}
where 
\begin{equation}
\mathcal{J}(\rho,\m,\mu) =  a^2 \int_0^1 \int_M \frac {g^{-1}(x)(\tilde{\m}(t,x),  \tilde{\m}(t,x))}{\tilde{\rho}(t,x)} \,\ud \nu(t,x) \, +  b^2 \int_0^1 \int_{M} \frac {\tilde{\mu}(t,x)^2}{\tilde{\rho}(t,x)} \,\ud \nu(t,x)
\end{equation}
over the set $(\rho,\m,\mu)$ satisfying $\rho \in \mathcal{M}([0,1] \times M)$, $\m \in (\Gamma_M^0([0,1] \times M,TM))^*$ which denotes the dual of time dependent continuous vector fields on $M$ (time dependent sections of the tangent bundle), $\mu \in \mathcal{M}([0,1] \times M)$
subject to the constraint
\begin{equation}\label{Eq:ContinuityEquation}
\int_0^1 \int_M \partial_t f\ud \rho +\int_0^1 \int_M \m(\nabla_x f)  - f\mu \ud \nu =  \int_M f(1,\cdot) \ud \rho_1 -  \int_M f(0,\cdot) \ud \rho_0
\end{equation} 
satisfied for every test function $f \in C^1([0,1]\times M,\R)$. Moreover, $\nu$ is chosen such that $\rho,\m,\mu$ are absolutely continuous with respect to $\nu$ and $\tilde{\rho},\tilde{\m},\tilde{\mu}$ denote their Radon-Nikodym derivative with respect to $\nu$. 
\end{definition}
\begin{remark}
Note that, in the previous definition, the divergence operator $\Div(\cdot)$ is defined by duality on the space of $C^1$ functions. In addition, since the functions in the integrand of formula \eqref{Extension1} are one homogeneous with respect to the triple of arguments $(\tilde{\rho},\tilde{\m},\tilde{\mu})$, the functional does not depend on the choice of $\nu$ which dominates the measures. Last, the Radon-Nikodym theorem applied to the measure $\m$ gives $\m =\tilde{\m} \nu$ where $\tilde{\m}$ is a measurable section of $T^*M$.
\end{remark}

\par This dynamical formulation enjoys most of the analytical properties of the initial Benamou-Brenier formulation \eqref{BB} and especially convexity. Moreover, $\on{WF}$ defines a distance on the space of nonnegative Radon measures which is continuous w.r.t. to the weak-* topology.
An important consequence is the existence of optimal paths in the space of time-dependent measures \cite{GeneralizedOT1} by application of the Fenchel-Rockafellar duality theorem.
Note in particular that the Hamiltonian formulation of the geodesic flow can be formally derived as
\begin{equation*}
\begin{cases}
\partial_t \rho(t,x) + \Div(\rho(t,x) \nabla_x q(t,x)) - 2q(t,x) \rho(t,x) = 0 \\
\partial_t q(t,x) + \|\nabla q(t,x)\|^2 + q(t,x)^2 = 0\,,
\end{cases}
\end{equation*}
where the second equation corresponds to the Hamilton-Jacobi equation \eqref{EqSimpleHJEquation}. 
In fact, not only analytical properties of standard optimal transport are conserved but also some interesting geometrical properties such as the Riemannian submersion highlighted by Otto, as explained in the introduction. More precisely, the group of diffeomorphisms can be replaced by a semi-direct product of group between $\Diff(M)$ and the space $C^\infty(M,\R_+^*)$ which is a group under pointwise multiplication. In addition, this group acts on the space of densities $\on{Dens}(M)$ and this action gives a Riemannian submersion between the group endowed with an $L^2$ type of metric, namely $L^2(M, \mathcal{C}(M))$ and the space of densities endowed with the Wasserstein-Fisher-Rao metric. The notation $\mathcal{C}(M)$ is the cone over $M$ defined in the next section \ref{Sec:ConeMetric}, it is the manifold $M \times \R_+^*$ endowed with the Riemannian metric given in Definition \ref{def:Cone}. Moreover, this semidirect product of groups is naturally identified as the automorphism group of the fibre bundle of half densities in section \ref{Sec:Automorphisms}.

\subsection{A cone metric}\label{Sec:ConeMetric}
To motivate the introduction of the cone metric, let us first discuss informally what happens for a particle of mass $m(t)$ at a spatial position $x(t)$ in a Riemannian manifold $(M,g)$ under the generalized continuity constraint \eqref{GCC}; If the control variables $v(t,x)$ and $\alpha(t,x)$ are Lipschitz, then the solution of the continuity equation with initial data $m(0)\delta_{x(0)}$ has the form $m(t) \delta_{x(t)}$ where $m(t) \in \R_+^*$ is the mass of the Dirac measure and $x(t) \in M$ its location; The system reads 
\begin{equation}
\begin{cases}
\dot{x}(t)= v(t,x(t))\\
\dot{m}(t) = \alpha(t,x(t))m(t)\,,
\end{cases} 
\end{equation}
which is directly obtained by duality since the flow map associated with $(v,\alpha)$ is well defined. This result would not hold if the vector field were not smooth enough, see \cite{AmbrosioFlowWeak}. 
Let us compute the action functional in the case where $\rho(t) = m(t)\delta_{x(t)}$. By the above result, $(v,\alpha)$ is completely determined at $(t,x(t))$ and it is sufficient to compute the action which reads $ \int_0^1a^2| v(x(t))|^2 m(t) + b^2\frac{\dot{m}(t)^2}{m(t)}\,\ud t $. Thus, considering the particle as a point in $M \times \R_+^*$, the Riemannian metric seen by the particle is $a^2 mg +  b^2\frac{\ud m^2}{m}$.  Therefore, it will be of importance to study this Riemannian metric $M \times \R_+^*$. Actually, this space is isometric to the standard Riemannian cone defined below.

\begin{definition}[Cone]\label{def:Cone}
Let $(M,g)$ be a Riemannian manifold. The cone over $M$ denoted by $\mathcal{C}(M)$ is the quotient space $\left(M \times \R_+\right) \, / \, \left(M \times \{ 0\}\right)$. The cone point $M \times \{ 0\}$ is denoted by $\mathcal{S}$.
The cone will be endowed with the metric $g_{\mathcal{C}(M)} \eqdef r^2g + \ud r^2$ defined on $M \times \R_+^*$ and $r$ is the variable in $\R_+^*$.
\end{definition}
The explicit formula for the distance on the Riemannian cone can be found in \cite{MetricGeometryBurago} and the isometry is given by the square root change of variable on the mass, as stated in the following proposition.
\begin{proposition}\label{Th:ConeDistance}
Let $a,b$ be two positive real numbers and $(M,g)$ be a Riemannian manifold.
The distance on $(M \times \R_+^*,a^2mg+\frac{b^2}{m}\ud m^2)$ is given by
\begin{equation}\label{ConeDistance}
d((x_1,m_1),(x_2,m_2))^2 = 4b^2\left(m_2 + m_1 -2\sqrt{m_1m_2} \cos\left(\frac{a}{2b}d_M(x_1,x_2) \wedge \pi\right)\right)\,,
\end{equation}
where the notation $\wedge$ stands for the minimum, that it $x \wedge y = \min(x,y)$ for $x,y \in \R$.
The space $(M \times \R_+^*,mg+\frac{1}{4m}\ud m^2)$ is isometric to $(\mathcal{C}(M),g_{\mathcal{C}(M)})$ by the change of variable $r = \sqrt{m}$. 
If $c$ is a unit speed geodesic for the metric $\frac{a^2}{4b^2} g$, an isometry
$S: \C \setminus \R_-  \to M\times \R_+^*$ is defined by  $S(re^{i\theta}) = (c(\theta),\frac{r^2}{4b^2})$.
\end{proposition}

In physical terms, it implies that mass can "appear" and "disappear" at finite cost. In other words, the Riemannian cone is not complete but adding the cone point, which represents $M\times \{0 \}$, to $M\times \R_+^*$ turns it into a complete metric space when $M$ is complete. Importantly, the distance associated with the cone metric \eqref{ConeDistance} is $1$-homogeneous in $(m_1,m_2)$. 
In the rest of the paper, unless explicitly mentioned, we consider the case $a = 1$ and $b = 1/2$.
We now collect known facts about Riemannian cones.
\begin{proposition}\label{Th:ConeConnection}
On the cone $\mathcal{C}(M)$, we denote by $e$ the vector field defined by $\frac{\partial}{\partial r}$. The Levi-Civita connection on $(M,g)$ will be denoted by $\nabla^g$. For a given vector field $X$ on $M$, define its lift as a vector field on $M \times \R_+^*$ by $\hat{X}(x,r)= (X(x),0)$. The Levi-Civita connection on $\mathcal{C}(M)$ denoted by $\nabla$ is given by
\begin{align*}
 \nabla_{\hat{X}}\hat{Y} = \widehat{\nabla_{X}^gY} - rg(X,Y) e \,,\,\,\,  \nabla_e e = 0
\text{ and } \nabla_e \hat{X} = \nabla_{\hat{X}} e = \frac{1}{r}\hat{X}\,.
\end{align*}
The curvature tensor $R$ on the cone satisfies the following properties,
\begin{equation}
R(\hat{X},e) = 0 \mbox{} \text{ and }R(\hat{X},\hat{Y})\hat{Z} = (R_g(X,Y)Z - g(Y,Z)X + g(X,Z)Y,0)
\end{equation} where $R_g$ denotes the curvature tensor of $(M,g)$.
Let $X,Y$ be two orthornormal vector fields on $M$,
\begin{equation}
K(\hat{X},\hat{Y}) =  K_g(X,Y) - 1
\end{equation}
where $K$ and $K_g$ denote respectively the sectional curvatures of $\mathcal{C}(M)$ and $M$.
\end{proposition}
\begin{proof}
Direct computations, see~\cite{Gallot1979}.
\end{proof}
Let us give simple comments on Riemannian cones: Usual cones, embedded in $\R^3$ are cones over $S_1$ of length less than $2\pi$. Although Riemannian cones over a segment in $\R$ are locally flat, the curvature still concentrates at the cone point. The cone over the sphere is isometric to the Euclidean space (minus the origin) and the cone over the Euclidean space has nonpositive curvature. In particular, the cone over $S_1$ is isometric to $\R^2 \setminus \{ 0\}$. We refer to \cite{MetricGeometryBurago} for more informations on cones from the point of view of metric geometry.
\par
We need the explicit formulas for the geodesic equations on the cone.
\begin{corollary}\label{ThCorrolaryGeodesicEquationOnTheCone}
The geodesic equations on the cone $\mathcal{C}(M)$ are given by
\begin{subequations}\label{Eq:GeodesicEquationCone}
\begin{align}
&\frac{D}{Dt}^g \dot{x} + 2\frac{\dot{r}}{r}\dot{x} = 0 \\
&\ddot{r} - rg(\dot{x},\dot{x}) = 0\,,
\end{align}
\end{subequations}
where $\frac{D}{Dt}^g$ is the covariant derivative associated with $(M,g)$.\\
Alternatively, the geodesic equations on $(M \times \R_+^*,a^2mg+\frac{b^2}{m}\ud m^2)$ can be written w.r.t. the initial "mass" coordinate as follows
\begin{subequations}\label{Eq:GeodesicEquationCone2}
\begin{align}
&\frac{D}{Dt}^g \dot{x} + \frac{\dot{m}}{m}\dot{x} = 0 \\
&\ddot{m} - \frac{\dot{m}^2}{2m} - \frac{a^2}{2b^2}g(\dot{x},\dot{x})m = 0\,.
\end{align}
\end{subequations}
\end{corollary}

Note that we used the isometry given in Proposition \ref{Th:ConeDistance} to derive the equations and in particular, we implicitly used the equality $4b^2 m = r^2$.
Since we have written the geodesic equations on the usual cone in polar coordinates, we used the square root of the "mass" coordinate, therefore we need to introduce below the space of square roots of densities to discuss the infinite dimensional setting.

\subsection{The automorphism group of the bundle of half-densities}\label{Sec:Automorphisms}
The cone can be seen as a trivial principal fibre bundle since $\mathcal{C}(M)$ is the direct product of $M$ with the group $\R_+^*$. Let us denote $p_M: \mathcal{C}(M) \mapsto M$ the projection on the first factor. The group $\R_+^*$ induces a group action on $\mathcal{C}(M)$ defined by $\lambda \cdot (x,\lambda') \eqdef (x,\lambda \lambda')$, for all $x \in M$ and $\lambda,\lambda' \in \R^*_+$. We now identify the trivial fibre bundle of half densities with the cone.
\begin{definition}\label{Def:HalfDensities}
Let $M$ be a smooth manifold without boundary and $(U_\alpha,u_\alpha)$ be a smooth atlas. The bundle of $s$-densities is the line bundle given by the following cocycle
\begin{align*}
&\Psi_{\alpha \beta}: U_\alpha \cap U_\beta \mapsto \on{GL}_1(\R) = \R^*\\
&\Psi_{\alpha \beta}(x) = |\on{det} (\ud (u_\beta \circ u_\alpha^{-1})(u_\alpha(x))|^s = \frac{1}{|\det(d(u_\alpha \circ u_\beta^{-1}))(u_\beta(x))|^s}\,\cdot
\end{align*}
\end{definition}
We denote by $\Dens_s(M)$ the set of sections of this bundle and we use $\Dens(M)$ instead of $\Dens_1(M)$, the space of densities. This definition shows that this fibre bundle is also a principal fibre bundle over $\R_+^*$ and it will be the point of view adopted in the rest of the paper.

On any smooth manifold $M$, the fibre bundle of $s$-densities is a trivial principal bundle over $\R_+^*$ since there exists a smooth positive density on $M$. Note that this trivialization depends on the choice of this reference positive density. If one chooses such a positive density, then the $1/2$-density bundle can be identified to the cone $\mathcal{C}(M)$. \textit{Let us fix the reference volume form to be the volume measure $\on{vol}$. By this choice, we identify $\Dens_{1/2}(M)$ with the set of sections of the cone $\mathcal{C}(M)$ in the rest of the paper}. 
Thus every element of $\Dens_{1/2}(M)$ is a section of the cone $\mathcal{C}(M)$. We are now interested in transformations that preserve the group structure. Namely, one can define 
\begin{equation}
\on{Aut}(\mathcal{C}(M)) = \left\{ \Phi \in \Diff(\mathcal{C}(M)) \, ; \, \Phi(x,r) = r \cdot \Phi(x,1) \text{ for all } r \in \R_+^* \right\}\,,
\end{equation}
which is the instantiation, in this particular case, of the definition of the automorphisms group of a principal fibre bundle. In other words, this is the subgroup of diffeomorphisms of the cone that preserve the group action on the fibers.
In particular, $\on{Aut}(\mathcal{C}(M))$ is a subgroup of $\Diff(\mathcal{C}(M))$. Of particular interest is the subgroup of $\on{Aut}(\mathcal{C}(M))$ which is defined as 
\begin{equation}
\on{Gau}(\mathcal{C}(M)) = \left\{ \Phi \in \on{Aut}(\mathcal{C}(M)) \, ; \, p_M \circ \Phi = \on{id}_M \right\}\,.
\end{equation}
The set $\on{Gau}(\mathcal{C}(M))$ is called the gauge group and it is a normal subgroup of $\on{Aut}(\mathcal{C}(M))$. We now consider the injection $\injs: \on{Diff}(M) \hookrightarrow \on{Aut}(\mathcal{C}(M))$ defined by $\injs(\varphi) = (\varphi,\on{id}_{\R_+^*})$. This is the standard situation of a semidirect product of groups between $i(\Diff(M))$ and $\on{Gau}(\mathcal{C}(M))$ since the following sequence is exact
\begin{equation}\label{Eq:ShortExactSequence}
\on{Gau}(\mathcal{C}(M)) \hookrightarrow \on{Aut}(\mathcal{C}(M)) \rightarrow \Diff(M)\,,
\end{equation}
where $\injs$ defined above provides an associated section of the short exact sequence and the projection from $ \on{Aut}(\mathcal{C}(M))$ onto $\Diff(M)$ is given by $\Phi \mapsto p_M \circ \Phi(x,1)$. 
Note that we could also have chosen the natural section associated to the natural bundle of half-densities. 
 As is well-known for a trivial principal bundle, $\on{Aut}(\mathcal{C}(M))$ is therefore equal to the semidirect product of group:
\begin{equation}
\on{Aut}(\mathcal{C}(M)) =  \Diff(M)\ltimes_\Psi \on{Gau}(\mathcal{C}(M))   \,,
\end{equation}
where $\Psi:\Diff(M) \mapsto \on{Aut}(\on{Gau}(\mathcal{C}(M)))$ is  defined by $\Psi(\varphi)(\lambda) = \varphi^{-1} \lambda \varphi$ being the associated inner automorphism of the group $\on{Gau}(\mathcal{C}(M))$, where the composition is understood as composition of diffeomorphisms of $\mathcal{C}(M)$. Being a trivial principal fibre bundle, the gauge group can be identified with the space of positive functions on $M$. Let us denote $\Lambda_{1/2}(M) \eqdef C^\infty(M,\R_+^*)$ which is a group under pointwise multiplication.  The subscript $1/2$ is a reminder of the fact that $\Lambda_{1/2}(M) $ is the gauge group of $\mathcal{C}(M)$, the bundle of $1/2$-densities. Note that we do not use the standard left action but, instead, a right action for the inner automorphisms as presented in \cite[Section 5.3]{Michor1993}, which fits better to our notations, although these two choices are equivalent. 
 The identification of $\Lambda_{1/2}$ with the gauge group $\on{Gau}(\mathcal{C}(M))$ is simply $\lambda \mapsto (\on{id}_{M},\lambda)$ where $(\on{id}_{M},\lambda):(x,m) \mapsto (x,\lambda(x)m)$. The group composition law is given by
\begin{equation}
(\varphi_1,\lambda_1) \cdot (\varphi_2,\lambda_2) = (\varphi_1 \circ \varphi_2,  (\lambda_1 \circ \varphi_2)  \lambda_2 )
\end{equation}
and the inverse is
\begin{equation}
(\varphi,\lambda)^{-1} = (\varphi^{-1},\lambda^{-1} \circ \varphi^{-1})\,.
\end{equation}
By construction, the group $\on{Aut}(\mathcal{C}(M))$ has a left action on the space $\Dens_{1/2}(M)$ as well as on $\Dens(M)$. The action on $\Dens(M)$ is explicitly defined by the map $\pi$ defined by 
\begin{align}& \pi: \left( \Diff(M) \ltimes_\Psi \Lambda_{1/2}(M) \right) \times \Dens(M) \mapsto \Dens(M) \nonumber\\
&\pi\left((\varphi,\lambda) , \rho \right) \eqdef  \varphi_* (\lambda^2 \rho)\,.
\end{align}\label{Eq:LeftAction}
For particular choices of metrics, this left action is a Riemannian submersion as detailed below. Note that we will use both automorphism group and semidirect product notations equally, depending on the context.

\subsection{A Riemannian submersion between the automorphism group and the space of densities}

The semidirect product of group $\Diff(M) \ltimes_\Psi \Lambda_{1/2}(M)$ will be endowed with the metric $L^2(M,\mathcal{C}(M))$ with respect to the reference measure on $M$. Let us recall it hereafter.
\begin{definition}[$L^2$ metric]\label{DefL2Metric}
Let $M$ be a manifold endowed with a measure $\mu$ and $(N,g)$ be a Riemannian manifold. Consider a measurable map $\varphi : M \to N$ and two measurable maps, $X,Y : M \mapsto TN$ such that $p_N \circ X = p_N \circ Y = \varphi$ where $p_N: TN \to N$ is the natural projection. 
Then, the $L^2$ Riemannian metric w.r.t. to the volume form $\mu$ and the metric $g$ at point $\varphi$ is defined by 
\begin{equation}
\langle X,Y \rangle_{\varphi} = \int_M g(\varphi(x))(X(\varphi(x)),Y(\varphi(x)))\,\ud \mu(x)\,.
\end{equation}
\end{definition}
This is probably the simplest type of (weak) Riemannian metrics on spaces of mappings and it has been studied in details in \cite{Ebin1970} in the case $L^2(M,M)$ and also in \cite{freed1989} where, in particular,  the curvature is computed for $L^2(M,N)$ for $N$ an other Riemannian manifold. 
Note in particular that this metric is \emph{not} the right-invariant metric $L^2$ on the semidirect product of groups as in \cite{Holm1998} or on the automorphism group which would lead to an EPDiff equation on a principal fibre bundle as developed in \cite{FGB2013}.
\begin{proposition}\label{Th:GeodesicEquationOnCone}
The geodesic equations on $\on{Aut}(\mathcal{C}(M))$ endowed with the metric $L^2(M,\mathcal{C}(M))$ with respect to the reference measure on $\nu$ are given by the geodesic equations on the cone \eqref{Eq:GeodesicEquationCone}, that is $\frac{D}{Dt}(\dot{\varphi},\dot{\lambda}) = 0$, or more explicitely 
\begin{subequations}\label{Eq:LagrangianFormulation}
\begin{align}
&\frac{D}{Dt}^g \dot{\varphi} + 2\frac{\dot{\lambda}}{\lambda}\dot{\varphi} = 0 \label{Eq:First}\\
&\ddot{\lambda} - \lambda g(\dot{\varphi},\dot{\varphi}) = 0\,. \label{Eq:Second}
\end{align}
\end{subequations}
\end{proposition}

\begin{proof}
This is a consequence of \cite{freed1989} or a direct adaptation of \cite[Theorem 9.1]{Ebin1970} to the case $L^2(M,\mathcal{C}(M))$ and Corollary \ref{ThCorrolaryGeodesicEquationOnTheCone}.
\end{proof}

We now state a crucial fact that arises from an elementary observation.
\begin{proposition}\label{Rem:TotallyGeodesic}
The automorphism group $\on{Aut}(\mathcal{C}(M))$  is totally geodesic in $\Diff(\mathcal{C}(M))$ for the $L^2(\mathcal{C}(M),\mathcal{C}(M))$ metric.
\end{proposition}
\begin{proof}
Note that the first equation \eqref{Eq:First} is $0$-homogeneous with respect to $\lambda$ and the second equation \eqref{Eq:Second} is  one homogeneous with respect to $\lambda$. This is a consequence of the fact that multiplication by positive reals acts as an affine isometry on $\mathcal{C}(M)$. Therefore, the path $\Phi(t): (x,r) \mapsto (\varphi(t)(x),\lambda(t)r)$ also satisfies the geodesic equation in $\Diff(\mathcal{C}(M))$ for the $L^2(\mathcal{C}(M),\mathcal{C}(M))$ metric.
\end{proof}
Note that this property does not depend on the measure on $\mathcal{C}(M)$ used in the definition of the space $L^2(\mathcal{C}(M),\mathcal{C}(M))$.

Let us first recall some useful notions. From the point of view of fluid dynamics, the next definition corresponds to the change of variable between Lagrangian and Eulerian formulations.
\begin{definition}[Right-trivialization]
Let $H$ be a group and a smooth manifold at the same time, possibly of infinite dimensions, the right-trivialization of $TH$ is the bundle isomorphism $\tau : T H \mapsto H \times T_{\Id} H$ defined by $\tau(h,X_h) \eqdef (h,dR_{h^{-1}} X_h)$, where $X_h$ is a tangent vector at point $h$ and $\mathcal{R}_{h^{-1}}: H \to H$ is the right multiplication by $h^{-1}$, namely, $R_{h^{-1}}(f)= f h^{-1}$ for all $f \in H$.
\end{definition}
In fluid dynamics, the right-trivialized tangent vector $dR_{h^{-1}} X_h$ corresponds to the spatial or Eulerian velocity  and $X_h$ is the Lagrangian velocity. Importantly, this right-trivialization map is continuous but not differentiable with respect to the variable $h$. Indeed, right-multiplication $R_h$ is smooth, yet left multiplication is continuous and usually not differentiable, due to a loss of smoothness.

\begin{example}
For the semi-direct product of groups defined above, we have 
\begin{equation}
\tau((\varphi,\lambda),(X_\varphi, X_\lambda)) = ((\varphi,\lambda),(X_\varphi \circ \varphi^{-1}, (X_\lambda \lambda^{-1}) \circ \varphi^{-1}))\,.
\end{equation}
We will denote by $(v,\alpha)$ an element of the tangent space of $T_{(\Id,1)}\Diff(M) \ltimes_\Psi \Lambda_{1/2}(M)$.
\end{example}
As an immediate consequence of Proposition \ref{Th:GeodesicEquationOnCone}, we write the geodesic equation in Eulerian coordinates.
\begin{corollary}[Geodesic equations in Eulerian coordinates]\label{Th:GeodesicEquationEulerianCone}
After right-trivialization, that is under the change of variable $v \eqdef \dot{\varphi} \circ \varphi^{-1}$ and $\alpha \eqdef \frac{\dot{\lambda}}{\lambda}\circ \varphi^{-1}$, the geodesic equations read
\begin{equation}
\begin{cases}
\,\dot{v} + \nabla_v v +2 \alpha v=0\\
\,\dot{\alpha} + \langle \nabla \alpha , v \rangle + \alpha^2 - g(v,v)  = 0\,.
\end{cases}
\end{equation}
\end{corollary}
Recall now the infinitesimal action associated with a group action.
\begin{definition}[Infinitesimal action]
For a smooth left action of $H$ a Lie group on a manifold $M$ and $q \in M$, the infinitesimal action is the map $T_{\Id} H \times M \mapsto TM $ defined by
\begin{equation}
\xi \cdot q \eqdef \left.\frac {\ud}{\ud t}\right|_{t=0} \left(\exp (\xi t) \cdot q \right) \in T_qM
\end{equation}
where $\cdot$ denotes the left action of $H$ on $M$ and $\exp(\xi t)$ is the Lie exponential, that is the solution to $\dot{h} = dR_h(\xi)$ and $h(0) = \Id$.
\end{definition}
\begin{example}\label{InfSemi}
For $\Diff(M) \ltimes_\Psi \Lambda_{1/2}(M)$ acting on $\Dens(M)$, the previous definition gives $(v,\alpha) \cdot \rho = -\Div(v\rho) + 2\alpha \rho$. Indeed, one has
$$ (\varphi(t), \lambda(t)) \cdot \rho = \on{Jac}(\varphi(t)^{-1}) (\lambda^2(t)\rho) \circ \varphi^{-1}(t)\,. $$
First recall that $\partial_t \varphi(t) = v \circ \varphi(t)$ and $\partial_t \lambda = \lambda(t) \alpha \circ \varphi(t)$.
Once evaluated at time $t=0$ where $\varphi(0) = \Id$ and $\lambda(0) = 1$, the differentiation with respect to $\varphi$ gives $-\Div(v\rho)$ and the second term $2\alpha \rho$ is given by the differentiation with respect to $\lambda$. 
\end{example}

We now recall the result of \cite[Claim of Section 29.21]{Michor2008b} in a finite dimensional setting. This result presents a standard construction to obtain Riemannian submersions from a transitive group action.
\begin{proposition}\label{Th:GeneralSubmersion}
Consider a smooth left action of Lie group $H$ on a manifold $M$ which is transitive and such that for every $\rho \in M$, the infinitesimal action $\xi \mapsto \xi \cdot \rho$ is a surjective map. Let $\rho_0  \in M$ and a Riemannian metric $G$ on $H$ that can be written as: 
\begin{equation}G(h)(X_h,X_h) =  g(h\cdot \rho_0)(dR_{h^{-1}} X_h,dR_{h^{-1}} X_h)
\end{equation} for $ g(h\cdot \rho_0)$ an inner product on $T_{\Id}H$. Let $X_\rho \in T_\rho M$ be a tangent vector at point $h\cdot \rho_0 = \rho \in M$, we define the Riemannian metric $\overline{g}$ on $M$ by
\begin{equation}\label{HLift}
\overline{g}(\rho)(X_\rho,X_\rho) \eqdef \min_{\xi \in T_{\Id} H} g(\rho)(\xi,\xi) \text{ under the constraint } X_\rho = \xi \cdot \rho \,.
\end{equation}
where $\xi = X_h \cdot h^{-1}$.

Then, the map $\pi_0: H \to M$ defined by $\pi_0(h) = h\cdot \rho_0$ is a Riemannian submersion of the metric $G$ on $H$ to the metric $\overline{g}$ on $M$. Moreover a minimizer $\xi$ in formula \eqref{HLift} is called an horizontal lift of $X_{\rho}$ at $\Id$.)
\end{proposition} 
The formal application of this construction in our infinite dimensional situation leads to the result, stated in \cite{GeneralizedOT2}:
\begin{proposition}\label{Th:RiemannianSubmersion}
Let $\rho_0 \in \Dens(M)$ and define the map 
\begin{align*}
&\pi_0:\on{Aut}(\mathcal{C}(M)) \to \Dens(M) \\
&\pi_0(\varphi, \lambda) = \varphi_* (\lambda^2 \rho_0)\,.
\end{align*}

Then, the map $\pi_0$ is a Riemannian submersion of the metric $L^2(M,\mathcal{C}(M))$ on the group $\on{Aut}(\mathcal{C}(M))$ to the Wasserstein-Fisher-Rao on the space of densities $\Dens(M)$.

The horizontal space and vertical space at $(\varphi,\lambda) \in \on{Aut}(\mathcal{C}(M))= \Diff(M) \ltimes_\Psi \Lambda_{1/2}(M)$
such that $\varphi_*(\lambda^2 \rho_0) = \rho$ are then defined by,
\begin{equation}
\on{Vert}_{(\varphi,\lambda)} = \left\{  \left(v,\alpha \right) \circ (\varphi,\lambda) \, ; \, (v,\alpha) \in \on{Vect}(M) \times C^\infty(M,\R) \text{ s.t. }  \Div(\rho v) = 2\alpha \rho \right\}\,,
\end{equation}
and 
\begin{equation}
\on{Hor}_{(\varphi,\lambda)} = \left\{  \left(\frac{1}{2}\nabla p,p\right) \circ (\varphi,\lambda) \, ; \, p \in C^\infty(M,\R) \right\}\,.
\end{equation}

\end{proposition}
Note that the minimization in $\eqref{HLift}$ is taken on an affine space of direction the vertical space whereas the minimizer is an element of the horizontal space.

Note also that the fibers of the submersion are right-cosets of the subgroup $H_0$ in $H$. 
The proof of the previous proposition is in fact given by the change of variables associated with right-trivialization.
Let $\rho_0$ be a reference density, the application of Proposition \ref{Th:GeneralSubmersion} gives
\begin{align}\label{EqCoreIdentity}
G(\varphi,\lambda)\!\left((X_\varphi,X_\lambda) , \!(X_\varphi,X_\lambda)\right)&\!= \!\int_{M} \!g(v,v) \rho \ud x + \int_{M}   \alpha^2  \rho \ud x \\
&\!=  \! \int_{M}\!  g( X_\varphi \circ \varphi^{-1},X_\varphi \circ \varphi^{-1}) \varphi_* (\lambda^2 \rho_0)\! \ud x + \! \int_{M}\! \!   (X_\lambda \lambda^{-1})^2 \circ \varphi^{-1}  \varphi_* (\lambda^2 \rho_0)\! \ud x\,\\
&\!= \!\int_{M} \!g(X_\varphi, X_\varphi) \lambda^2 \rho_0 \ud x +   \int_{M}  X_\lambda^2 \, \rho_0 \ud x\,.
\end{align}
Therefore, the metric $G$ is the $L^2(M,\mathcal{C}(M))$ metric with respect to the density $\rho_0$. 
Moreover, in this particular situation, the horizontal lift is a minimizer of \eqref{HLift}.  
\begin{proposition}[Horizontal lift]\label{HorizontalLiftFormulation}
Let $\rho \in \Dens^s(\Omega)$ be a smooth density and $X_\rho \in H^s(\Omega,\R)$ 
be a tangent vector at the density $\rho$.
The horizontal lift at $(\Id, 1)$ of $X_\rho$ is given by $(\frac12\nabla \Phi,\Phi)$ where $\Phi$ is the solution to the elliptic partial differential equation:
\begin{equation}\label{EllipticEquation}
-\frac12\Div(\rho  \nabla \Phi)  + 2\Phi \rho = X_\rho\,.
\end{equation}
By elliptic regularity, the unique solution $\Phi$ belongs to $H^{s+1}(M)$.
\end{proposition}
To prove Proposition \ref{HorizontalLiftFormulation}, remark that equation \eqref{EllipticEquation} is the first order condition of the minimization problem \eqref{HLift} where the term $X_{\rho} $ reads in this case $X_{\rho}=\xi \cdot \rho = (v,\alpha) \cdot \rho = - \Div(\rho v) + 2\alpha \rho$.

\par
A direct application of this Riemannian submersion viewpoint is the formal computation of the sectional curvature of the Wasserstein-Fisher-Rao in this smooth setting by applying O'Neill's formula see \cite{GallotHulinLafontaine}. To recall it hereafter, we need the Lie bracket of right-invariant vector fields on $\Diff(M) \ltimes_\Psi \Lambda_{1/2}(M)$.
\begin{proposition}
Let $(v_1,\alpha_1)$ and $(v_2,\alpha_2)$ be two tangent vectors at identity in $\Diff(M) \ltimes_\Psi \Lambda_{1/2}(M)$. Then,
 \begin{equation}
[(v_1,\alpha_1),(v_2,\alpha_2)] = \left( [v_1,v_2] , \nabla \alpha_1\cdot v_2 - \nabla \alpha_2\cdot v_1 \right)\,,
\end{equation}
where $ [v_1,v_2]$ denotes the Lie bracket of vector fields.
 \end{proposition}
Note that the application of this formula to horizontal vector fields gives $[(\frac 12 \nabla \Phi_1,\Phi_1),(\frac 12 \nabla \Phi_2,\Phi_2)]
 = (\frac14 [\nabla \Phi_1,\nabla \Phi_2],0)$.

\begin{proposition}\label{prop-curvature}
Let $\rho$ be a smooth positive density on $M$ and $X_1, X_2$ be two orthonormal tangent vectors at $\rho$ and $\xi_{\Phi_1},\xi_{\Phi_2}$ be their corresponding right-invariant horizontal lifts on the group.
If O'Neill's formula can be applied, the sectional curvature of $\Dens(M)$ at point $\rho$ is given by, 
\begin{equation}\label{SectionalCurvatureONeill}
K(\rho)(X_1,X_2) = \int_\Omega k(x,1)(\xi_1(x),\xi_2(x)) w(\xi_1(x),\xi_2(x)) \rho(x) \ud \nu(x) + \frac 34 \left\| [\xi_1,\xi_2]^V \right\|^2
\end{equation}
where  
$$w(\xi_1(x),\xi_2(x)) = g_{\mathcal{C}(M)}(x,1)(\xi_1(x),\xi_1(x))g_{\mathcal{C}(M)}(x)(\xi_2(x),\xi_2(x))-\left(g_{\mathcal{C}(M)}(x,1)(\xi_1(x),\xi_2(x))\right)^2$$ and $ [\xi_{\Phi_1},\xi_{\Phi_2}]^V$ denotes the vertical projection of $[\xi_{\Phi_1},\xi_{\Phi_2}]$ at identity, $\| \cdot \|$ denotes the norm at identity and $k(x,1)$ is the sectional curvature of the cone at point $(x,1)$ in the directions $(\xi_1(x),\xi_2(x))$.
\end{proposition}
This computation is only formal and we will not attempt here to give a rigorous meaning to this formula similarly to what has been done in \cite{lott2008some} for the $L^2$ Wasserstein metric. Yet, it has interesting consequences: the curvature of the space of densities endowed with the $\on{WF}$ metric is always greater or equal than the curvature of the cone $\mathcal{C}(M)$. In particular, it is non-negative if the curvature of $(M,g)$ is bigger than $1$, as a consequence of Proposition \ref{Th:ConeConnection}.


\section{The $H^{\Div}$ right-invariant metric on the diffeomorphism group}
\label{Sec:EulerArnold}

In this section, we summarize known results on the $H^{\Div}$ right-invariant metric on the diffeomorphism group. 
We now define the $H^{\Div}$ right-invariant metric.
\begin{definition}\label{DefinitionHdiv}
Let $(M,g)$ be a Riemannian manifold and $\on{Diff}^s(M)$ be the group of diffeomorphisms which belong to $H^s(M)$ for $s > d/2 + 1$. The right-invariant $H^{\Div}$ metric, implicitely dependent on two positive real parameters $a,b$, is defined by
\begin{equation}\label{Eq:Metric}
G_\varphi(X_\varphi,X_\varphi) = \int_M a^2| X_\varphi \circ \varphi^{-1} |^2 + b^2\Div(X_\varphi \circ \varphi^{-1})^2 \ud \! \on{vol} \,.
\end{equation}
\end{definition}
The Euler-Arnold equation in one dimension (that is on the circle $S^1$ for instance) is the well-known Camassa-Holm equation (actually when $a=b=1$):
\begin{equation} \label{Eq:CH}
a^2\partial_{t} u - b^2\partial_{txx} u  +3 a^2\partial_{x} u \,u  - 2b^2\partial_{xx} u \,\partial_x u - b^2\partial_{xxx} u \,u = 0\,.
\end{equation}
On a general Riemannian manifold $(M,g)$, the equation can be written as, with $n = a^2 u^\flat  + b^2 \ud \delta u^\flat$,
\begin{equation}\label{Eq:HdivGeometricForm}
\partial_t n + a^2\left(\Div(u) u^\flat +  \ud \langle u , u \rangle + \iota_u \ud u^\flat\right)  + b^2\left( \Div(u) \ud \delta u^\flat + \ud \iota_u \ud \delta u^\flat\right) = 0\,,
\end{equation}
where the notation $\flat$ corresponds to lowering the indices. More precisely, if $u \in \chi(M)$ then $u^\flat$ is the $1$-form defined by $v \mapsto g(u,v)$. The notation $\delta$ is the formal adjoint to the exterior derivative $\ud$ and $\iota$ is the insertion of vector fields which applies to forms.

\par \textbf{On the well-posedness of the initial value problem.} Although the theorem below is not stated in this particular form in \cite{Ebin1970}, this result can be seen as a byproduct of their results as explained in \cite[Theorem 4.1]{MisolekPreston2010}. For similar smoothness results in the case of smooth diffeomorphisms, we refer the reader to \cite[Theorem 3]{Constantin2003}.
\begin{theorem}\label{Th:EbinMarsden}
 On $\on{Diff}^s(S_1)$ for $s \geq 2$ integer, the $H^{1}$ right-invariant metric is a smooth and weak Riemannian metric. Moreover, if $s\geq 3$, the exponential map is a smooth local diffeomorphism on $T\on{Diff}^s(S_1)$.
\end{theorem}
Global well-posedness does not hold in one dimension since there exist smooth initial conditions for the Camassa-Holm equation such that the solutions blow up in finite time.

In higher dimensions, the initial value problem has been studied by Michor and Mumford \cite[Theorem 3]{MumfordMichorHdiv}. This is not a direct result of \cite{Ebin1970} since the differential operator associated to the metric is not elliptic. They prove that the initial value problem on the space of vector fields is locally well posed for initial data in a Sobolev space of high enough order. Although the proof could probably be adapted to the case of a Riemannian manifold, in that case, the result of local well posedness is not proven yet. 

\par \textbf{On the metric properties of the  $H^{\Div}$ right-invariant metric.}
Michor and Mumford already had the following non-degeneracy result in \cite{Michor2005}.
\begin{theorem}[Michor and Mumford]\label{Th:MM}
The distance on $\on{Diff}(M)$ induced by the $H^{\Div}$ right-invariant metric is non-degenerate. Namely, between two distinct diffeomorphisms the infimum of the lengths of the paths joining them is strictly positive.
\end{theorem}
Due to the presence of blow up in the Camassa-Holm equation, metric completeness does not hold since it would imply geodesic completeness, that is global well posedness. However, it is still meaningful to ask whether geodesics are length minimizing for short times. Since the Gauss lemma is valid in a strong $H^s$ topology, this ensures that geodesics are length minimizing among all curves that stay in a $H^s$ neighborhood, see also \cite{Constantin2003}.
However, this is \emph{not} enough to prove that the associated geodesic distance is non degenerate since an almost minimizing geodesic can escape this neighborhood for arbitrarily small energy. This is what happens for the right-invariant metric $H^{1/2}$ on the circle $S_1$ where the metric is degenerate although there exists a smooth exponential map similarly to our case in 1D, see \cite{Escher12}.


\section{A Riemannian submanifold point of view on the $H^{\Div}$ right-invariant metric}\label{Sec:Submanifold}

The starting point of this section is the following simple proposition whose proof is omitted.
\begin{proposition}
Consider a Riemannian submersion constructed as in Proposition \ref{Th:GeneralSubmersion}. Let $H_0$ be the isotropy subgroup of $\rho_0$, then, considering $H_0$ as a Riemannian submanifold of $H$ and denoting $G_{H_0}$ its induced metric, $G_{H_0}$ is a right-invariant metric on $H_0$. 
\end{proposition}

The Riemannian submersion $\pi_0:\on{Aut}(\mathcal{C}(M)) \mapsto \Dens(M)$ defined in Proposition \ref{Th:RiemannianSubmersion} enables to study the equivalent problem to the incompressible Euler equation.
The fiber of the Riemannian submersion at $\on{vol}$ is $\pi_0^{-1}(\{ \on{vol} \})$ and it will be denoted by $\on{Aut}_{\on{vol}}(\mathcal{C}(M))$, it therefore corresponds to the group $H_0$ in the previous proposition. More explicitely, we have
\begin{equation}\label{Eq:Fiber1}
\pi_0^{-1}(\{ \on{vol} \}) = \{ (\varphi,\lambda) \in \on{Aut}(\mathcal{C}(M))\, : \, \varphi_*(\lambda^2 \on{vol}) = \on{vol} \}\,.
\end{equation}
The constraint $\varphi_*(\lambda^2 \on{vol}) = \on{vol}$ can be made explicit as follows
\begin{equation}\label{Eq:Fiber2}
\on{Aut}_{\on{vol}}(\mathcal{C}(M)) = \{ (\varphi,\sqrt{\on{Jac}(\varphi)}) \in \on{Aut}(\mathcal{C}(M))\,: \, \varphi \in \Diff(M) \}\,.
\end{equation}
Note that this isotropy subgroup can be identified with the group of diffeomorphims of $M$ since the map $\varphi \mapsto (\varphi,\sqrt{\on{Jac}(\varphi)})$ is also a section of the short exact sequence \eqref{Eq:ShortExactSequence}.
This shows that there is a natural identification between $\Diff(M)$ and $\on{Aut}_{\on{vol}}(\mathcal{C}(M))$.
Now, the vertical space at point $(\varphi,\sqrt{\on{Jac}(\varphi)}) \in \on{Aut}_{\on{vol}}(\mathcal{C}(M)) $ is
\begin{equation}\label{Eq:VerticalSpace1}
\on{Ker}\left(d\pi_0(\varphi,\sqrt{\on{Jac}(\varphi)})\right)= \{ (v,\alpha) \comment{\circ}  (\varphi,\sqrt{\on{Jac}(\varphi)}) \,: \, \Div v = 2\alpha \ \}\,,
\end{equation}
and equivalently
\begin{equation}\label{Eq:VerticalSpace2}
\on{Ker}\left(d\pi_0(\varphi,\sqrt{\on{Jac}(\varphi)})\right)= \left\{ \left(v,\frac12 \Div v\right) \comment{\circ}   (\varphi,\sqrt{\on{Jac}(\varphi)}) \,: \, v \in \on{Vect}(M)  \right\}\,.
\end{equation}

It is now possible to apply equation \eqref{EqCoreIdentity} to obtain the explicit formula for the right-invariant metric on $\on{Aut}_{\on{vol}}(\mathcal{C}(M))$. 
The metric $L^2(M,\mathcal{C}(M))$ on $\on{Aut}(\mathcal{C}(M))$ restricted to $\Diff(M) \simeq \on{Aut}_{\on{vol}}(\mathcal{C}(M)) $ reads 
\begin{equation}
G_\varphi(X_\varphi,X_\varphi) = \int_M |v|^2 \ud\!\on{vol}  + \frac 14 \int_M |\Div v|^2 \ud\!\on{vol}\,,
\end{equation}
where $ v = X_\varphi \circ \varphi^{-1}$.
Therefore, on $\Diff(M) \simeq \on{Aut}_{\on{vol}}(\mathcal{C}(M)) $, the induced metric is a right-invariant $H^{\Div}$ metric. In other words, we have
\begin{theorem}\label{Th:IsometricInjection}
By its identification with $\on{Aut}_{\on{vol}}(\mathcal{C}(M))$, the diffeomorphism group endowed with the $H^{\on{div}}$ right-invariant metric, see Definition \ref{DefinitionHdiv}, is isometrically embedded in $L^2(M,\mathcal{C}(M))$.
\end{theorem}

As a straightforward application, we retrieve theorem \ref{Th:MM}.
\begin{corollary}
The distance on $\Diff(M)$ with the right-invariant metric $H^{\Div}$ is non degenerate.
\end{corollary}
\begin{proof}
Let $\varphi_0,\varphi_1 \in \Diff(M)$ be two diffeomorphisms and $c$ be a path joining them. The length of the path $c$ for the right-invariant metric $H^{\Div}$ is equal to the length of the lifted path $\tilde{c}$ in $\on{Aut}(\mathcal{C}(M))$.
Since $L^2(M,\mathcal{C}(M))$ is a Hilbert manifold, the length of the path $\tilde{c}$ is bounded below by the length of the geodesic joining the natural lifts of $\varphi_0$ and $\varphi_1$ in $L^2(M,\mathcal{C}(M))$. Therefore, it leads to
\begin{equation}
d_{H^{\Div}}(\varphi_0,\varphi_1) \geq d_{L^2(M,\mathcal{C}(M))}\left( (\varphi_0,\sqrt{\on{Jac}(\varphi_0)}),(\varphi_1,\sqrt{\on{Jac}(\varphi_1)})  \right)\,.
\end{equation}
If $d_{H^{\Div}}(\varphi_0,\varphi_1) =0$ then $d_{L^2(M,\mathcal{C}(M))}\left( (\varphi_0,\sqrt{\on{Jac}(\varphi_0)}),(\varphi_1,\sqrt{\on{Jac}(\varphi_1)})  \right) = 0$ which implies $\varphi_0 = \varphi_1$.
\end{proof}

\begin{remark}[The Fisher-Rao metric]
In \cite{Khesin2013}, it is shown that the $\dot{H}^1$ right-invariant metric descends to the Fisher-Rao metric on the space of densities. Let us explain why this situation differs from ours: It is well known that a left action of a group endowed with a right-invariant metric induces on the orbit a Riemannian metric for which the action is a Riemannian submersion.
However, Khesin et al. do not consider a left action, but a right action on the space of densities: More precisely, if a reference density $\rho$ is chosen, the map they considered is 
\begin{align*}
 \on{Diff}(M) &\to \Dens(M) \\
\varphi & \mapsto \varphi^{*}\rho \,.
\end{align*}
Obviously, this situation is equivalent to a left action of a group of diffeomorphisms endowed with a left-invariant metric. In such a situation, the descending metric property has to be checked \cite[Proposition 2.3]{Khesin2013}. 
\par 
Their result can be read from our point of view: The $\dot{H}^1$ metric is $\frac 14 \int_M |\Div v|^2 \ud \mu$ and it corresponds to the case where $a=0$. It thus leads to a degenerate metric on the group. Viewed in the ambient space $L^2(M,\mathcal{C}(M))$, the projection on the bundle component is a (pseudo-) isometry from $L^2(M,\mathcal{C}(M))$ (endowed with this pseudo-metric) to the space of densities since $a=0$. Moreover, on the space of densities which lie in the image of the projection, that is, the set of probability densities, the projected metric is the Fisher-Rao metric.
\end{remark}

We now use the identification between $\Diff(M)$ endowed with the right-invariant $H^{\Div}$ metric and $\on{Aut}_{\on{vol}}(\mathcal{C}(M))$ as a submanifold of $\on{Aut}(\mathcal{C}(M))$ and write the geodesic equations in this setting. As is standard for the incompressible Euler equation, the constraint is written in Eulerian coordinates and the corresponding geodesic are written hereafter.
\begin{theorem}
The geodesic equations on the fiber $\on{Aut}_{\on{vol}}(\mathcal{C}(M))$ as a Riemannian submanifold of $\on{Aut}(\mathcal{C}(M))$ endowed with the metric $L^2(M,\mathcal{C}(M))$ can be written in Lagrangian coordinates 
\begin{equation}\label{Eq:GeodEquationFinal}
\begin{cases}
\frac{D}{Dt}\dot{\varphi} + 2\frac{\dot{\lambda}}{\lambda}\dot{\varphi} = -\frac12 \nabla^g p \circ \varphi \\
\ddot{\lambda} - \lambda g(\dot{\varphi},\dot{\varphi}) = -  \lambda p \circ \varphi\,,
\end{cases}
\end{equation}
with a function $P: M \to \R$.
\\
In Eulerian coordinates, the geodesic equations read
\begin{equation}
\begin{cases}
\dot{v} + \nabla^g_v v + 2v \alpha = -\frac12 \nabla^g p  \\
\dot{\alpha} + \langle \nabla \alpha , v \rangle + \alpha^2 - g(v,v) = -p\,,
\end{cases}
\end{equation}
where $\alpha = \frac{\dot{\lambda}}{\lambda} \circ \varphi^{-1}$ and $v = \partial_t \varphi \circ \varphi^{-1}$.
\end{theorem}

This submanifold point of view leads to a generalization of \cite[Theorem A.2]{KhesinCurvature} on the sectional curvature of $\Diff(M)$ which has been computed and studied in \cite{KhesinCurvature}. The authors show that the curvature of $\Diff(S_1)$ can be written using the Gauss-Codazzi formula and they show the explicit embedding in a semi-direct product of groups similar to our situation.
\par
As mentioned above, we consider $\Diff(M)$ as a submanifold of $L^2(M,\mathcal{C}(M))$. The second fundamental form can be computed as in the case of the incompressible Euler equation. 

\begin{proposition}
Let $U,V$ be two smooth right-invariant vector fields on $\on{Aut}(\mathcal{C}(M))$ that can be written as $U(\varphi,\lambda) = (u,\alpha) \circ (\varphi,\lambda)$ and $V(\varphi,\lambda) = (v,\beta) \circ (\varphi,\lambda)$.
The second fundamental form for the isometric embedding $\Diff(M) \hookrightarrow L^2(M,\mathcal{C}(M))$ is
\begin{equation}\label{Eq:II}
\Two(U,V) = \left( -\frac12 \nabla p \circ \varphi, - \lambda p \circ \varphi \right) \,,
\end{equation}
where $p = (2\on{Id}-\frac12 \Delta)^{-1} A(\nabla_{(u,\alpha)}(v,\beta))$ is the unique solution of the elliptic PDE \eqref{EllipticEquation}
 \begin{equation}
(2\on{Id}-\frac12 \Delta)(p) = A(\nabla_{(u,\alpha)}(v,\beta)) \,,
\end{equation} 
where $A(w,\gamma) \eqdef \Div(w) -2 \gamma$. Using the explicit expression of $\nabla_{(u,\alpha)}(v,\beta)$ the elliptic PDE  reads
 \begin{equation} \label{Eq:HLiftForII}
(2\on{Id}-\frac12\Delta)(p) =  \Div(\nabla_uv + \beta u + \alpha v) -2\langle \nabla \beta,u\rangle + 2g(u,v) - 2\alpha \beta \,.
\end{equation} 
\end{proposition}

\begin{proof}
By right-invariance of the metric, it suffices to treat the case $(\varphi,\lambda) = \on{Id}$. The orthogonal projection is the horizontal lift defined in Proposition \ref{HorizontalLiftFormulation}. Therefore, we compute the infinitesimal action of $\nabla_{(u,\alpha)}(v,\beta)$ on the volume form which is given by the linear operator $A$ and we consider its horizontal lift $(-\frac12\nabla p, -p)$ given by Proposition \ref{HorizontalLiftFormulation}. 
By right-invariance, the orthogonal projection at $(\varphi,\lambda)$ is given by $\left( -\frac12\nabla p \circ \varphi, -\lambda p \circ \varphi \right)$.
\par
From Proposition \ref{Th:ConeConnection}, one has
\begin{equation}
\nabla_{(u,\alpha)}(v,\beta) = \left( \nabla_uv + \beta u + \alpha v, \langle \nabla \beta,u\rangle - g(u,v) + \alpha \beta \right)\,,
\end{equation}
and Formula \eqref{Eq:HLiftForII} follows. 
\end{proof}

We can then state the Gauss-Codazzi formula applied to our context.
\begin{proposition}
Let $U,V$ be two smooth right-invariant vector fields on $\on{Aut}_{\on{vol}}(\mathcal{C}(M))$ written as $U(\varphi,\lambda) = (u,\alpha) \circ (\varphi,\lambda)$ and $V(\varphi,\lambda) = (v,\beta) \circ (\varphi,\lambda)$.
The sectional curvature of $\Diff(M)$ endowed with the right-invariant $H^{\Div}$ metric is
\begin{equation}
\langle R_{\on{Diff}(M)}(U,V)V,U \rangle =\langle R_{L^2(M,\mathcal{C}(M))}(U,V)V,U \rangle + \langle \Two(U,U) , \Two(V,V) \rangle -  \langle \Two(U,V) , \Two(U,V) \rangle\,,
\end{equation}
where $\Two$ is the second fundamental form \eqref{Eq:II} and 
\begin{equation}
\langle R_{L^2(M,\mathcal{C}(M))}(U,V)V,U \rangle = \int_M \langle R_{\mathcal{C}(M)}(u,v)v,u \rangle \circ (\varphi,\lambda) \ud \mu\,,
\end{equation}
where $(\varphi,\lambda) \in \on{Aut}(\mathcal{C}(M))$.
\end{proposition}

\begin{proof}
The only remaining point is the computation of the sectional curvature of $L^2(M,\mathcal{C}(M))$ which is done in Freed and Groisser's article \cite{freed1989}.
\end{proof}

Note that the sectional curvature of $L^2(M,\mathcal{C}(M))$ vanishes if $M = S_n$ since $\mathcal{C}(M) = \R^{n+1}$, which is the case for the one-dimensional Camassa-Holm equation. However, for $M = T_n$, $n\geq 2$, the flat torus, the sectional curvature of $\mathcal{C}(M)$ is non-positive and bounded below by $-1$ and thus the sectional curvature of $L^2(T_n,\mathcal{C}(T_n))$ is non-positive.


\section{Applications}\label{Sec:Applications}

The point of view developed above provides an example of an isometric embedding of the group of diffeomorphisms endowed with the right-invariant $H^{\Div}$ metric in an $L^2$ space such as $L^2(M,N)$, here with $N = \mathcal{C}(M)$. 
In this section, we develop two applications of this point of view. The first one consists in rewriting the Camassa-Holm equation as particular solutions of the incompressible Euler equation on the cone; the results hold in higher dimensions for the geodesics of the $H^{\Div}$ metric. The second application is about minimizing properties of solutions of the Camassa-Holm equation and its generalization with $H^{\Div}$. We prove that, under mild conditions, smooth solutions are length minimizing for short times.

\subsection{The Camassa-Holm equation}\label{Sec:CHAsEuler}
Let us consider the following Camassa-Holm equation, 
\begin{equation} \label{Eq:CH23}
\begin{cases}
\partial_{t} u - \frac14 \partial_{txx} u +3 \partial_{x} u \,u  - \frac12\partial_{xx} u \,\partial_x u - \frac14\partial_{xxx} u \,u = 0\,\\
\partial_t \varphi(t,x) = u(t,\varphi(t,x))\,.
\end{cases}
\end{equation}
With respect to the standard Camassa-Holm equation, this equation has different coefficients that are chosen here to simplify the discussion. Unless otherwise mentioned, all the results still apply to the standard formulation of the equation.
For such a choice of coefficients, the cone construction $\mathcal{C}(S_1)$ is isometric to $\R^2 \setminus \{ 0\}$ with the Euclidean metric. Following Theorem \ref{Th:IsometricInjection}, we have the isometric injection
\begin{align}
\mathcal{M}: \Diff(S_1) &\to \on{Aut}(\mathcal{C}(S_1)) \subset L^2(S_1,\R^2)\\
\varphi & \mapsto (\varphi,\sqrt{\varphi'}) = \sqrt{\varphi'} e^{i\varphi}\,.
\end{align}
Then, solutions of the Camassa-Holm equation are geodesic for the flat metric $L^2(S_1,\R^2)$ on the constrained submanifold of maps $(\varphi,\lambda)$ defined by the constraint $\varphi' = \lambda^2$. 
Note that the map $\mathcal{M}$ is very similar to a Madelung transform which maps solutions of the Schr\"odinger equation to solutions of a compressible Euler type of hydrodynamical equation. The geodesic equation on $ \on{Aut}(\mathcal{C}(S_1))$ reads
\begin{equation}\label{Eq:GeodesicSubmanifoldSimple2}
\begin{cases}
\ddot{\varphi} + 2\frac{\dot{\lambda}}{\lambda}\dot{\varphi} = - \frac12 \partial_x p \circ \varphi \\
\ddot{\lambda} - \lambda \dot{\varphi}^2 = -  \lambda  p \circ \varphi\,,
\end{cases}
\end{equation}
where $p: S_1 \to \R$.
Formula \eqref{Eq:GeodesicSubmanifoldSimple2} looks similar to the incompressible Euler equation in Lagrangian coordinates. However, this geodesic equation is apparently written on the space of maps $S_1 \mapsto \mathcal{C}(S_1)$. Since $\on{Aut}(\mathcal{C}(S_1)) \subset \Diff(\mathcal{C}(S_1))$, it can be expected to be a geodesic equation on the group of diffeomorphism of the cone. Indeed, we have
\begin{theorem}
Solutions to the Camassa-Holm equation on $S_1$
\begin{equation}
\partial_{t} u - \frac14 \partial_{txx} u +3 \partial_{x} u \,u  - \frac12\partial_{xx} u \,\partial_x u - \frac14\partial_{xxx} u \,u = 0\,
\end{equation}
are mapped to solutions of the incompressible Euler equation on $\R^2 \setminus \{ 0\}$ for the density $ \rho = \frac{1}{r^4} \on{Leb}$, that is
\begin{equation}\label{Eq:EulerSimple2}
\begin{cases}
\dot{v} + \nabla_v v = -\nabla P\,,\\
\nabla \cdot (\rho v) = 0\,,
\end{cases}
\end{equation}
\begin{equation*}
\mbox{ by the map } : 
\left[ \begin{array}{c}
u\, :S_1 \to \R \\
\theta \mapsto u(\theta)
\end{array} \right] \mapsto 
\left[\begin{array}{c}
v\,: S_1\times \R_*^+ = \mathcal{C}(S_1)  \to \R^2 \\
(\theta,r) \mapsto \left(u(\theta),\frac{r}{2} \partial_x u(\theta)\right)
\end{array}\right]
\end{equation*}
\end{theorem}

\begin{proof}
We show that $\mathcal{M}(\varphi)$ provides solutions to the incompressible Euler equation written in Lagrangian coordinates. 
The second equation in \eqref{Eq:GeodesicSubmanifoldSimple2} being linear in $\lambda$ and the first equation being $0$ homogeneous in $\lambda$, the geodesic equations can be rewritten as
\begin{equation}\label{Eq:GeodesicSubmanifold}
\begin{cases}
\ddot{\varphi} + 2\frac{\dot{\lambda}}{\lambda}\dot{\varphi} = - \frac12\partial_x p \circ \varphi \\
\ddot{\lambda}r - \lambda r \dot{\varphi}^2 = - \lambda r  p \circ \varphi\,.
\end{cases}
\end{equation}
Thus, the map $\Phi(t) : (x,r) \mapsto (\varphi(t,x),\lambda(t,x)r)$ satisfies
\begin{equation}
\ddot{\Phi}(t)(x,r) = - \nabla \Psi_p(t) \circ \Phi(t) \,,
\end{equation}
where $\Psi_p(x,r) = \frac12 r^2 p (x)$. This formulation is close to the incompressible Euler equation, however, we need to check if the density $\rho(r,\theta) = 1/r^3 \ud r \ud \theta$ is preserved by pull-back by $\Phi$, or equivalently due to the group structure, by pushforward. We first compute the Jacobian matrix, recalling that $\lambda=\sqrt{\partial_x \varphi}$,
\begin{equation*}
D\Phi(x,r) = \begin{pmatrix} \partial_x \varphi & 0 \\ \frac{\partial_{xx} \varphi}{2\sqrt{\partial_x \varphi}} & \sqrt{\partial_x \varphi}\end{pmatrix}\,,
\end{equation*}
whose determinant is $(\partial_x \varphi)^{3/2} $.
We now compute the pushforward
\begin{align*}
\on{Jac}(\Phi) \rho \circ \Phi(x,r) &= 1/(r\,\sqrt{\partial_x \varphi})^3 \on{Jac}(\Phi) \\
& = 1/(r \, \sqrt{\partial_x \varphi})^3 (\partial_x \varphi)^{3/2} = \frac{1}{r^3} = \rho(x,r) \,.
\end{align*}
This proves the result in Lagrangian coordinates. To get the formulation in the theorem, one differentiates the map $\Phi$ at identity which gives $(u,\frac r2 \partial_x u)$ for the vector field in polar coordinates.
\end{proof}

{\color{black} 
\begin{remark}[About the blow-up]
At this point, a natural question is about the difference between global well-posedness of incompressible Euler in 2D, whereas the Camassa-Holm equation has a well understood blow-up. Of course, there is no contradiction since the density for which the CH equation is similar to Euler has a singularity at zero, which allows for unbounded vorticity although we did not check this possibility. 
In a similar direction, we can cite \cite{2017Elgindi}, since the authors mention that the singularity comes "from the vorticity amplification due to the presence of a density gradient".
Note also that the typical situation of blow-up of the CH equation in the case of colliding peakons can be understood in this situation as the quantity $\sqrt{\partial_x \varphi}$ goes to zero in finite time.
\end{remark}
}
The second application consists in showing that smooth solutions of the Camassa-Holm equation are length minimizing for short times.  
\begin{theorem}[Smooth solutions to the Camassa-Holm equation \eqref{Eq:CH23} are length minimizing for short times.]
Let $(\varphi(t),\lambda(t))$ be a smooth solution to the geodesic equations \eqref{Eq:CH23} (in the formulation \eqref{Eq:GeodesicSubmanifoldSimple2}) on the time interval $[t_0,t_1]$.
If $(t_1-t_0)^2|\langle w, \nabla^2 \Psi_{p}(x,r) w \rangle| < \pi^2 \| w \|^2$ holds for all $t\in [t_0,t_1]$ and $(x,r) \in \mathcal{C}(S_1)$ and $w \in T_{(x,r)} \mathcal{C}(S_1)$, then for every smooth curve $(\varphi_0(t),\lambda_0(t)) \in \on{Aut}_{\on{vol}}(\mathcal{C}(S_1))$ satisfying $  (\varphi_0(t_i),\lambda_0(t_i)) = (\varphi(t_i),\lambda(t_i))$ for $i = 0,1$ one has
\begin{equation}
\int_{t_0}^{t_1} \| (\dot{\varphi},\dot{ \lambda}) \|^2 \, \ud t \leq \int_{t_0}^{t_1} \| (\dot{\varphi}_0,\dot{ \lambda }_0) \|^2 \, \ud t \,,
\end{equation}
with equality if and only if the two paths coincide on $[t_0,t_1]$.

\end{theorem}

\begin{remark}
This result only applies to this choice of coefficients and for other choices of coefficients the result still holds in an $L^\infty$ neighborhood of the geodesic. This is done in the more general case of $H^{\Div}$ in the next section. Since the proof is a direct adaptation of Brenier's \cite{BRENIER200355} and it is simple in this particular case, we include it hereafter. It also helps to understand the proof in the general case of a Riemannian manifold.
\end{remark}

\begin{proof}
To alleviate notations, we denote $g_t = (\varphi(t),\lambda(t))$ and $h_t = (\varphi_0(t),\lambda_0(t))$. 
Since $p$ can be chosen with zero mean, $\Psi_p(x,r) = \frac12 r^2 p (x)$ and $g_t = (\varphi(t),\sqrt{\on{Jac}(\varphi(t))})$, by direct integration, for every $t \in [t_0,t_1]$
\begin{equation}
\int_{S_1} \Psi_p(g_t(x)) \ud x = 0\,.
\end{equation}
The same equality holds for $h_t$.
Let $s \in [0,1] \mapsto c(t,s,x)$ be a two parameters ($t\in [t_0,t_1]$ and $x \in {S_1}$) smooth family of geodesics on $\R^2$ such that $c(t,0,x) = g_t(x)$ and $c(t,1,x) = h_t(x)$ for every $t \in [t_0,t_1]$ and $x \in {S_1}$. 
Let us define $J(t,s,x) = \partial_t c(t,s,x)$, we have
\begin{equation}
J(t,0,x) = \partial_t g_t(x) \text{ and } J(t,1,x) = \partial_t h_t(x)\,.
\end{equation}
Now, the result we want to prove can be reformulated as,
\begin{equation}
\int_{t_0}^{t_1} \int_{S_1} \| J(t,0,x) \|^2 \ud t \ud x \leq \int_{t_0}^{t_1} \int_{S_1} \| J(t,1,x) \|^2 \ud t \ud x
\end{equation}
with equality if and only if for almost every $x$, it holds $g_t(x)  = h_t(x)$ for all $t \in [t_0,t_1]$.
Using a second-order Taylor expansion of $\Psi_p(c(t,s,x))$ with respect to $s$ at $s = 0$ and denoting by $C \eqdef \sup_{t \in [t_0,t_1]}\sup_{x \in {S_1}} \|\nabla^2 \Psi_{p}(x)\| $, we have, 
\begin{align*}
& \Psi_p(h_t(x)) -  \Psi_p(g_t(x)) - \langle \nabla \Psi_p(c(t,0,x)), \partial_s c(t,0,x) \rangle  \leq \frac{C}{2} \int_{0}^{1}\|\partial_s c(t,s,x)\|^2 \ud s \,.
\end{align*}
We will integrate in time $t$ and apply the one dimensional Poincaré inequality in the $t$ variable
\begin{equation}
\int_{t_0}^{t_1} \|\partial_s c(t,s,x)\|^2 \ud t \leq \frac{C(t_1 - t_0)^2}{2\pi^2} \int_{t_0}^{t_1} | \partial_t \|\partial_s c(t,s,x) \| |^2 \ud t\,,
\end{equation}
for every $s,x$. Since $c(t,0,x)$ is a solution of the Camassa-Holm equation, one has $\partial_{tt} c = - \nabla \Psi_p(t)$.
Thus, we have, integrating in time 
\begin{align*}
\int_{t_0}^{t_1} \Psi_p(h_t(x)) -  \Psi_p(g_t(x)) + \langle \partial_{tt}c(t,0,x) , \partial_s c(t,0,x) \rangle \ud t \leq \frac{C(t_1 - t_0)^2}{2\pi^2} \int_{t_0}^{t_1}  \int_0^1 | \partial_t \|\partial_s c(t,s,x) \| |^2 \ud s \ud t
 \,.
\end{align*}
We also have $| \partial_t \|\partial_s c(t,s,x) \| |^2 \leq  \|\partial_{ts} c(t,s,x) \| ^2$. Then,
integrating over $S_1$, the two first terms on the l.h.s. vanish and integrating by part in time, we get
\begin{equation}
\int_{t_1}^{t_2} \int_{S_1} - \langle \partial_{t}c(t,0,x) , \partial_{st} c(t,0,x) \rangle \ud t \leq \frac{C(t_1 - t_0)^2}{2\pi^2} \int_{t_0}^{t_1}  \int_{S_1} \int_0^1  \|\partial_{ts} c(t,s,x) \| ^2 \ud s \ud x\ud t
 \,,
\end{equation}
where we used the fact that $\partial_s c(t,s,x)$ is constant in $s$ since the geodesics on the plane are straight lines.
Writing $f(s) =\frac 12  \int_{t_0}^{t_1} \int_{S_1} \| J(t,s,x) \|^2 \ud t $, we want to prove $f(1) \geq f(0)$ and we have $$-f'(0) \leq \frac{C(t_1 - t_0)^2}{2\pi^2} \int_{t_0}^{t_1} \int_{S_1} \int_0^1 \| \partial_sJ(t,s,x) \|^2 \ud s \ud x\ud t\,.$$
Therefore, the result is proven if we can show that for some $\varepsilon > 0$
\begin{equation}
f(1) - f(0) - f'(0) \geq \varepsilon \int_{t_0}^{t_1} \int_{S_1} \int_0^1 \| \partial_s J(t,s,x) \|^2 \ud s \ud x\ud t\,.
\end{equation}
We have $f(1)- f(0) - f'(0) = \int_0^1 (1-s) f''(s) \ud s$ and here $f''(s) = \int_{t_0}^{t_1} \int_{S_1} \| \partial_s J(t,s,x) \|^2 \ud t \ud x$ since $\partial_{ss} J = 0$ because $\R^2$ has vanishing curvature, and also $\partial_s J = \on{cste}(t,x)$, a constant w.r.t. $s$. Hence, we get 
\begin{equation}
f(1)- f(0) - f'(0) =\frac 12  \int_{t_0}^{t_1} \int_{S_1} \int_0^1 \| \partial_sJ(t,s,x) \|^2 \ud s \ud x\ud t\,.
\end{equation}
Therefore, $$f(1) - f(0) \geq \left(\frac 12 - \frac{C(t_1 - t_0)^2}{2\pi^2}\right)  \int_{t_0}^{t_1}\int_{S_1} \int_0^1 \| \partial_sJ(t,s,x) \|^2 \ud s \ud x\ud t\,,$$
which is nonnegative if $t_1 - t_0 \leq \frac{\pi}{\sqrt{C}}$.
\end{proof}

 \begin{remark}
The condition on the Hessian is satisfied for smooth paths, see Remark \ref{RemConditionOnHessian}.
Moreover, similarly to Brenier's proof, the constant is sharp since the rotation at unit speed is a particular solution of the Camassa-Holm equation for which the Hessian is equal to $1$ and it stops being a minimizer at the angle $\pi$.
 \end{remark}
\subsection{The $H^{\Div}$ case in higher dimensions}\label{Sec:GeneralHdiv}
In the general case, we are left with the geometry of the cone, and therefore, the map $\mathcal{M}$ maps solutions of the geodesic equation on the diffeomorphisms group for the right-invariant $H^{\Div}$ metric to solutions of the incompressible Euler equation on the $\mathcal{C}(M)$ for a density which has a singularity at the cone point. 
In the general case, the geodesic equation is written as
\begin{equation}\label{Eq:GeodesicSubmanifold2}
\begin{cases}
\frac{D}{Dt}\dot{\varphi} + 2\frac{\dot{\lambda}}{\lambda}\dot{\varphi} = -\frac12\nabla^g p \circ \varphi \\
\ddot{\lambda}r - \lambda r g(\dot{\varphi},\dot{\varphi}) = -  \lambda r p \circ \varphi\,.
\end{cases}
\end{equation}
Viewing the automorphisms $(\varphi,\lambda)$ of the cone as diffeomorphisms of the cone, the geodesic equation is close to incompressible Euler equations, with the difference that the automorphisms do not preserve the Riemannian volume measure on $\mathcal{C}(M)$ but another density which has a singularity at the cone point. 

\begin{theorem}
On the group of diffeomorphisms of the cone, the geodesic equation can be written 
\begin{equation}\label{Eq:CompactFormGeodesic}
\frac{D}{Dt} (\dot{\varphi},\dot{\lambda r}) = - \nabla \Psi_p \circ (\varphi,\lambda r) \,,
\end{equation}
where $\Psi_p(x,r) \eqdef \frac12 r^2p(x)$. Moreover, the diffeomorphisms of $\mathcal{C}(M)$ $(\varphi,\lambda)$ preserve the measure $\tilde{\nu} \eqdef r^{-3}\ud r \ud \! \on{vol}$.
\\
In other words, a solution $ (\varphi,\lambda)$ of \eqref{Eq:CompactFormGeodesic} is a solution of the incompressible Euler equation for the density $r^{-3-d} \ud \! \on{vol}_{\mathcal{C}(M)}$ 
where $\ud \! \on{vol}_{\mathcal{C}(M)}$ is the volume form on the cone $\mathcal{C}(M)$ and $d$ is the dimension of $M$.
\end{theorem}

\begin{proof}
The geodesic equations \eqref{Eq:GeodesicSubmanifold2} can be rewritten in the form \eqref{Eq:CompactFormGeodesic} since a direct computation gives $\nabla \Psi_p = (\frac12 \nabla^g p,  r p)$.
\par
The only remaining point is that $(\varphi,\lambda)$ preserves the measure $r^{-3}\ud \nu \ud r$ on $\mathcal{C}(M)$, if the relation $\lambda = \sqrt{\on{Jac}(\varphi)}$ holds. Indeed, the volume form $r^{\th} \ud \nu \ud r$ is preserved by $(\varphi, \lambda)$ if and only if the following equality is satisfied $ (\lambda r)^\th \lambda \on{Jac}(\varphi)  = r^\th$, equivalently $\lambda^{\th + 3} = 1$. It is the case if and only if $\th = -3$.
\end{proof}

In particular, this theorem underlines that $\on{Aut}_{\on{vol}}(\mathcal{C}(M)) = \on{Aut}(\mathcal{C}(M)) \, \cap \, \on{SDiff}_{\tilde{\nu}}(\mathcal{C}(M))$. In remark \ref{Rem:TotallyGeodesic}, we mentioned that $\on{Aut}(\mathcal{C}(M))$ is a totally geodesic subspace of $\on{Diff}(\mathcal{C}(M))$, which explains the fact that the geodesic equation on $\on{Aut}_{\on{vol}}(\mathcal{C}(M))$ is actually a geodesic equation on $\on{SDiff}_{\tilde{\nu}}(\mathcal{C}(M))$. We illustrate this situation in Figure \ref{Fig}. 
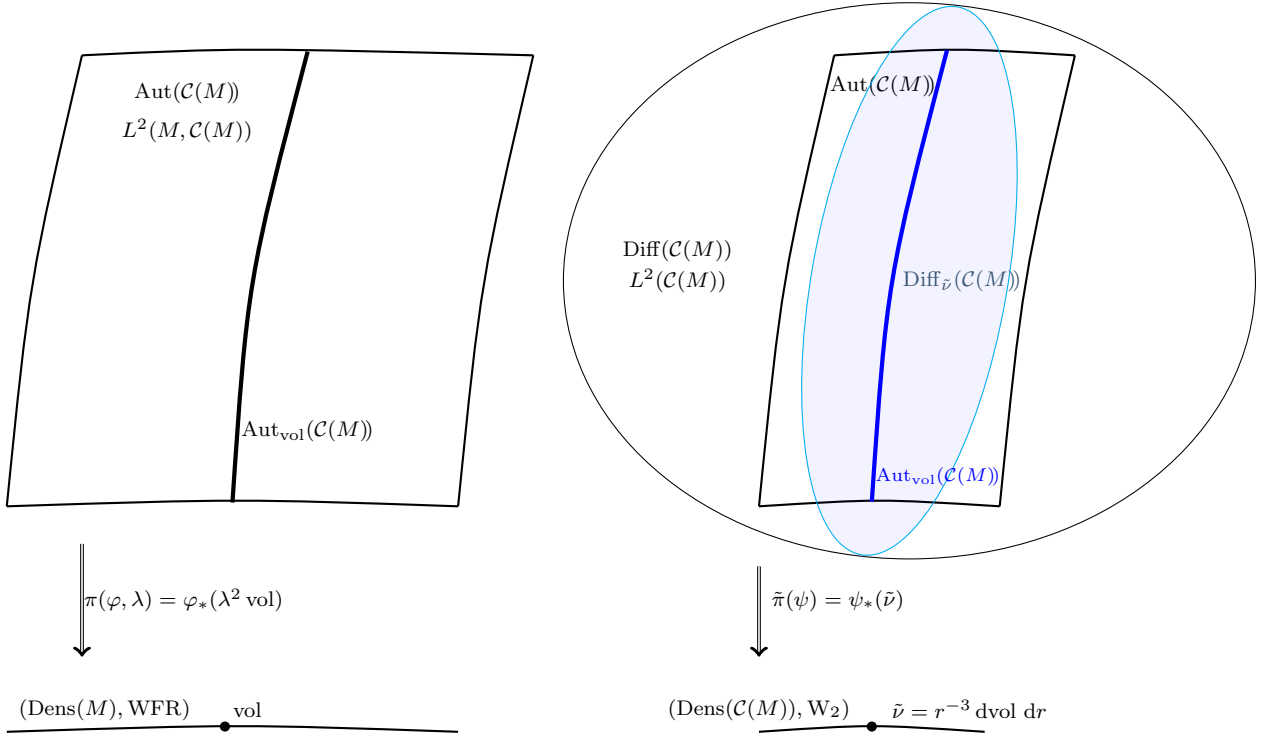
\begin{figure}
\centering
\begin{tikzpicture}
\draw [thick] (0,0) .. controls (0.3,3) .. (1,6); 
\draw [thick] (0,0) .. controls (3.1,0.1).. (6,0);
\draw [thick] (1,6) .. controls (3.1,6.1) .. (7,6);
\draw [thick] (6,0) .. controls (6.3,3) .. (7,6);
\node at (4,1){\footnotesize{$\on{Aut}_{\on{vol}}(\mathcal{C}(M)\!)$}};
\draw [ultra thick] (3,0.05) .. controls (3.2,3) .. (4,6.05);
\node at (1.3,-2.7) {\footnotesize{$(\Dens(M),\on{WFR})$}};
\node at (3.2,-2.7) {\footnotesize{$\on{vol}$}};
\node at (2.4,5.5) {\footnotesize{$\on{Aut}(\mathcal{C}(M)\!)$}};
\node at (2.4,5) {\footnotesize{$L^2(M,\mathcal{C}(M))$}};
\draw [thick] (0,-3) .. controls (3.1,-2.9).. (6,-3);
\draw[fill] (2.9,-2.93) circle [radius=0.06];
\draw[->,double] (1,-0.5) -- (1,-2); 
\node at (2.35,-1.25){\footnotesize{$\pi(\varphi,\lambda) = \varphi_*(\lambda^2\on{vol})$}};


\draw [thick] (10,0) .. controls (10.3,3) .. (11,6); 
\draw [thick] (10,0) .. controls (11.6,0.1).. (13.2,0);
\draw [thick] (11,6) .. controls (12.6,6.1) .. (14.2,6);
\draw [thick] (13.2,0) .. controls (13.5,3) .. (14.2,6);
\node at (11.65,5.6){\footnotesize{$\on{Aut}(\mathcal{C}(M)\!)$}};
\node at (8.92,3.4) {\footnotesize{$\Diff(\mathcal{C}(M))$}};
\node at (8.92,3) {\footnotesize{$L^2(\mathcal{C}(M))$}};
\draw [ultra thick,blue] (11.5,0.06) .. controls (11.7,3) .. (12.5,6.07);
\node at (10,-2.7) {\footnotesize{$(\Dens(\mathcal{C}(M)),\on{W}_2)$}};
\node at (12.8,-2.7) {\footnotesize{$\tilde{\nu} = r^{-3} \ud \! \on{vol}  \ud r$}};
\draw [thick] (10,-3) .. controls (11.5,-2.9).. (13,-3);
\draw[fill] (11.5,-2.93) circle [radius=0.06];
\draw (12,3) circle [x radius = 4.6,y radius = 3.7];  
\draw[rotate around={-10:(12,3)},draw = cyan, fill=blue, opacity=0.05] (12,3) circle [x radius = 1.3,y radius = 3.7];  
\draw[rotate around={-10:(12,3)},draw = cyan] (12,3) circle [x radius = 1.3,y radius = 3.7];
\node[color={rgb:red,135;green,206;blue,280}] at (12.7,3) {\footnotesize{$\on{Diff}_{\tilde{\nu}}(\mathcal{C}(M)\!)$}};
\draw[->,double] (10,-0.8) -- (10,-2); 
\node[blue] at (12.38,0.4){\scriptsize{$\on{Aut}_{\on{vol}}(\mathcal{C}(M)\!)$}};
\node at (11.05,-1.25){\footnotesize{$\tilde{\pi}(\psi) = \psi_*(\tilde{\nu})$}};
\end{tikzpicture}

\caption{On the left, the picture represents the Riemannian submersion between $\on{Aut}(\mathcal{C}(M))$ and the space of positive densities on $M$ and the fiber above the volume form is $\on{Aut}_{\on{vol}}(\mathcal{C}(M))$.
On the right, the picture represents the automorphism group $\on{Aut}(\mathcal{C}(M))$ isometrically embedded in $\Diff(\mathcal{C}(M))$ and the intersection of
$\on{Diff}_{\tilde{\nu}}(\mathcal{C}(M))$ and $\on{Aut}(\mathcal{C}(M))$ is equal to $\on{Aut}_{\on{vol}}(\mathcal{C}(M))$.}
\label{Fig}
\end{figure}

%
%
%
%
%

The same result holds on more general Riemannian manifolds.
We propose a straightforward generalization of Brenier's proof \cite{BRENIER200355} in the case of Euler equation to a Riemannian setting. Note that, to our knowledge, no previous result was available on minimizing $H^{\Div}$ geodesics. In the worst case of our theorem, we require only an $L^\infty$ bound on the Jacobian and on the diffeomorphism.
\begin{theorem}\label{ThMinimizingGeodesicsHdiv}
Let $(\varphi(t),\lambda(t))$ be a smooth solution to the geodesic equations \eqref{Eq:CompactFormGeodesic} on the time interval $[t_0,t_1]$.
If $(t_1-t_0)^2|\langle w, \nabla^2 \Psi_{p}(x,r) w \rangle| < \pi^2 \| w \|^2$ holds for all $t\in [t_0,t_1]$ and $(x,r) \in \mathcal{C}(M)$ and $w \in T_{(x,r)} \mathcal{C}(M)$, then for every smooth curve $(\varphi_0(t),\lambda_0(t)) \in \on{Aut}_{\on{vol}}(\mathcal{C}(M))$ satisfying $  (\varphi_0(t_i),\lambda_0(t_i)) = (\varphi(t_i),\lambda(t_i))$ for $i = 0,1$ and the condition $(*)$, one has
\begin{equation}
\int_{t_0}^{t_1} \| (\dot{\varphi},\dot{ \lambda}) \|^2 \, \ud t \leq \int_{t_0}^{t_1} \| (\dot{\varphi}_0,\dot{ \lambda }_0) \|^2 \, \ud t \,,
\end{equation}
with equality if and only if the two paths coincide on $[t_0,t_1]$.
\par 
Define $\delta_0 \eqdef \min \{r(x,t) \,: \, \text{injectivity radius at }  (\varphi(t,x),\lambda(t,x)) \}$,
then the condition $(*)$ is: 
\begin{enumerate}
\item If the sectional curvature of $\mathcal{C}(M)$ can assume both signs or if $\on{diam}(M) \geq \pi$, there exists $\delta$ satisfying $0<\delta < \delta_0$ such that the curve $(\varphi_0(t),\lambda_0(t))$ has to belong to a  $\delta$-neighborhood of $(\varphi(t),\lambda(t))$, namely
\begin{equation}\nonumber
d_{\mathcal{C}(M)}\left((\varphi_0(t,x),\lambda_0(t,x)),(\varphi(t,x),\lambda(t,x))\right)) \leq \delta
\end{equation}
for all $(x,t) \in  M \times [t_0,t_1]  $ where $d_{\mathcal{C}(M)}$ is the distance on the cone.
\item If $\mathcal{C}(M)$ has non positive sectional curvature, then, for every $\delta < \delta_0$, there exists a short enough time interval on which the geodesic will be length minimizing. 
\item If $M = S_d(1)$, the result is valid for every path $(\dot{\varphi}_0,\dot{\lambda}_0)$.
\end{enumerate}
\end{theorem}

\begin{remark}\label{RemConditionOnHessian}
Importantly, the condition on the Hessian is not empty, i.e. it is fulfilled in our case of interest: Indeed, when $p$ is a $C^2$ function on $M$, the Hessian of $\Psi_p(x,r) = \frac12r^2p(x)$ is, in the orthonormal basis $\partial_r,\frac{1}{r}e_1,\ldots,\frac{1}{r}e_d$ where $e_1,\ldots,e_d$ is an orthornormal basis of $T_xM$
\begin{equation}
\nabla^2 \Psi_{p}(x,r) = \begin{pmatrix} \frac12\nabla^2 p(x)  & \nabla p(x) \\ \nabla p^T(x) & p(x) \end{pmatrix}\,,
\end{equation}
where $\nabla p$ is the gradient of $p$ in the orthornormal basis $e_1,\ldots,e_d$.
Since $p$ is smooth and $M$ is compact, the Hessian of $p$ is bounded uniformly on $\mathcal{C}(M)$.
\end{remark}

The proof is postponed in Appendix. The generalization of Brenier's proof that we propose is not completely satisfactory in positive curvature or, in the case of negative curvature, because of the injectivity radius bound. In the former case, the constructed interpolating paths have to pass through the cone point and therefore these paths $c(t,s,x)$ are not smooth any longer w.r.t. $s$ and thus Jacobi fields are not smooth a priori. These two limitations could probably be overcome using a different strategy than a geodesic homotopy between the two diffeomorphisms. We actually conjecture that the result holds true without the boundedness assumption.

\section{Future directions}

In this article, we have presented the geometric link between the Camassa-Holm equation and the new $L^2$ Wasserstein optimal transport metric between positive Radon measures. 
We presented an isometric embedding of the group of diffeomorphism group endowed with the right-invariant $H^{\Div}$ metric in the space $L^2(M,\mathcal{C}(M))$. This isometric embedding enables to rewrite the Camassa-Holm equation, via a Madelung transform, as an incompressible Euler equation on the cone. In other words, the Camassa-Holm equation is a geodesic flow on $\on{Aut}_{\on{vol}}(\mathcal{C}(M))$ for the $L^2$ metric.
As an application, this has also led to a result on the minimizing property of geodesics. 
The point of view developed in this paper can be taken to address the variational problem of shortest path for the $H^{\Div}$ metric in the sense of Brenier \cite{Brenier1999,Brenier2013}, which appears to be a non-trivial problem.
Following Brenier, we will investigate elsewhere the uniqueness of the pressure as in \cite{Brenier1993}. This isometric embedding and the polar factorization theorem opens the way to design new numerical simulations of variational  solutions of the Camassa-Holm equation, in the direction of \cite{Merigot2016, Merigot2015}.
\par
Following the point of view developed in this article, we plan to rewrite other fluid dynamic equations as  geodesic equations on a submanifold of a space of maps endowed with an $L^2$ norm. The result may have, as shown for the Camassa-Holm equation, interesting analytical consequences.

\appendix

\section{Proof of Theorem \ref{ThMinimizingGeodesicsHdiv}}

\begin{proof}
To alleviate notations, we denote $g_t = (\varphi(t),\lambda(t))$ and $h_t = (\varphi_0(t),\lambda_0(t))$. 
Since $p$ can be choose with zero mean, $\Psi_p(x,r) = \frac12 r^2 p (x)$ and $g_t = (\varphi(t),\sqrt{\on{Jac}(\varphi(t))})$, by direct integration, for every $t \in [t_0,t_1]$, 
\begin{equation}
\int_{M} \Psi_p(g_t(x)) \ud x = 0\,.
\end{equation}
The same equality holds for $h_t$.
\par
Let $s \in [0,1] \mapsto c(t,s,x)$ be a two parameters ($t\in [t_0,t_1]$ and $x \in M$) family of geodesics on $\mathcal{C}(M)$ such that $c(t,0,x) = g_t(x)$ and $c(t,1,x) = h_t(x)$ for every $t \in [t_0,t_1]$ and $x \in M$.
This family of geodesics is uniquely 
defined if one considers balls which do not intersect the cut locus. Uniformity of the radius of the balls can be obtained since $[t_0,t_1] \times M$ is compact, which defines $\delta_0$. Consequently, the family of curves $c(t,s,x)$ is a smooth family of geodesics, at least as smooth as $g_t(x)$ and $h_t(x)$ are with respect to the parameters $t,x$. 
Since $\partial_t c(t,s,x)$ is a variation of geodesics, it is a Jacobi field as a function of $s$. Thus, we will use the notation 
$J(t,s,x) = \partial_t c(t,s,x)$. Consequently, we have 
\begin{equation}
J(t,0,x) = \partial_t g_t(x) \text{ and } J(t,1,x) = \partial_t h_t(x)\,.
\end{equation}
Now, the result we want to prove can be reformulated as,
\begin{equation}
\int_{t_0}^{t_1} \int_M \| J(t,0,x) \|^2 \ud t \ud x \leq \int_{t_0}^{t_1} \int_M \| J(t,1,x) \|^2 \ud t \ud x
\end{equation}
with equality if and only if for almost every $x$, it holds $g_t(x)  = h_t(x)$ for all $t \in [t_1,t_2]$.
We now use a second-order Taylor expansion of $\Psi_p(c(t,s,x))$ with respect to $s$ at $s = 0$. Denoting by $C \eqdef \sup_{t \in [t_0,t_1]}\sup_{x \in M} \|\nabla^2 \Psi_{p_t}(x)\| $, we have, 
\begin{align*}
& \Psi_p(h_t(x)) -  \Psi_p(g_t(x)) - \langle \nabla \Psi_p(c(t,0,x), \partial_s c(t,0,x) \rangle  \leq \frac{C}{2} \int_{0}^{1}\|\partial_s c (t,s,x)\|^2 \ud s \,.
\end{align*}
Now, one has that $\partial_s c(t,s,x)$ vanishes at $t =t_0$ and $t = t_1$. We can therefore apply Poincaré inequality to $\|\partial_s  c(t,s,x) \|$ to obtain 
\begin{equation}
\int_{t_0}^{t_1} \|\partial_s c(t,s,x) \|^2 \ud t \leq \frac{C(t_1 - t_0)^2}{2\pi^2} \int_{t_0}^{t_1} | \partial_t \|\partial_s  c(t,s,x) \| |^2 \ud t\,.
\end{equation}
Since $ \partial_t \|\partial_s  \| = \frac{1}{\|\partial_s c \|} \langle \nabla_t \partial_s c, \partial_s c\rangle$, we have the inequality $| \partial_t \|\partial_s  \| |  \leq \| \nabla_t \partial_s c \|$ and we get, exchanging derivatives,
\begin{equation}\label{Eq:PoincareInequality}
\int_{t_0}^{t_1} \|\partial_s c(t,s,x) \|^2 \ud t \leq \frac{C(t_1 - t_0)^2}{2\pi^2} \int_{t_0}^{t_1} \| \dot{J} (t,s,x) \|^2 \ud t\,,(t,0,x)
\end{equation}
where $\dot{J}$ is the covariant derivative of $J$ with respect to $s$. We thus have
\begin{align*}
& \int_{t_0}^{t_1} \Psi_p(c(t,1,x)) -  \Psi_p(c(t,0,x)) - \langle \nabla \Psi_p(c(t,0,x)), \partial_s c(t,0,x) \rangle  \leq \frac{C(t_1 - t_0)^2}{2\pi^2} \int_{t_0}^{t_1}  \int_{0}^{1}  \| \dot{J}(t,s,x) \|^2  \ud s \ud t  \,.
\end{align*}
However, $g_t(x) = c(t,0,x)$ is a solution of $\nabla_t \partial_t c(t,0,x) = - \nabla \Psi_p(t,0,x)$, therefore, an integration by part w.r.t. $t$ leads to
\begin{align*}
& \int_{t_0}^{t_1} \Psi_p(c(t,1,x)) -  \Psi_p(c(t,0,x)) - \langle \partial_t c(t,0,x), \nabla_t \partial_s c(t,0,x) \rangle \ud t \leq \frac{C(t_1 - t_0)^2}{2\pi^2} \int_{t_0}^{t_1}  \int_{0}^{1}  \| \dot{J}(t,s,x) \|^2 \ud s   \ud t  \,.
\end{align*}
Last, integrating over $M$ and exchanging once again covariant derivatives gives
\begin{align*}
& \int_{t_0}^{t_1} \int_M - \langle J(t,0,x), \dot{J}(t,0,x)  \rangle \ud x \ud t \leq \frac{C(t_1 - t_0)^2}{2\pi^2} \int_{t_0}^{t_1} \int_M \int_0^1 \| \dot{J}(t,s,x) \|^2 \ud s \ud x\ud t \,.
\end{align*}
Writing $f(s) =\frac 12 \int_{t_0}^{t_1} \int_M \| J(t,s,x) \|^2 \ud t $, we want to prove $f(1) \geq f(0)$ and we have $$-f'(0) \leq \frac{C(t_1 - t_0)^2}{2\pi^2} \int_{t_0}^{t_1} \int_M \int_0^1 \| \dot{J}(t,s,x) \|^2 \ud s \ud x\ud t\,.$$
Therefore, the result is proven if we can show
\begin{equation}
f(1) - f(0) - f'(0) \geq \varepsilon \int_{t_0}^{t_1} \int_M \int_0^1 \| \dot{J}(t,s,x) \|^2 \ud s \ud x\ud t\,.
\end{equation}
The left hand side can be reformulated using $f(1) - f(0) - f'(0) = \int_0^1 (1-s)f''(s) \ud s$ as
\begin{equation}\label{Eq:IneqToProve}
\int_{t_0}^{t_1} \int_M \int_0^1 (1-s) ( \| \dot{J} \|^2 - \langle R(\partial_s c,J)J,\partial_sc\rangle ) \ud s \ud x\ud t \geq \varepsilon \int_{t_0}^{t_1} \int_M \int_0^1 \| \dot{J} \|^2 \ud s \ud x\ud t\,,
\end{equation} 
with $\varepsilon = \frac{C(t_1 - t_0)^2}{2\pi^2}$.
\par
We now need to distinguish between two cases, the first one being when $ \int_{t_0}^{t_1} \int_M \int_0^1  \| \dot{J} \|^2  \ud s \ud x \ud t \geq 1$.
In this case, we use the inequality 
\begin{equation}\label{Eq:SimpleIneq}\| J(t,1,x) \|^2 \leq 2\| J(t,0,x)\|^2 + 2\int_0^1 \|\dot{J}(t,s,x) \|^2 \ud s\,,
\end{equation} 
in order to get
\begin{equation}
- \int_{t_0}^{t_1} \int_M \int_0^1 (1-s) \langle R(\partial_s c,J)J,\partial_sc\rangle \ud s \ud x\ud t  \leq  \delta^2 \int_{t_0}^{t_1} \int_M\int_0^1 K_{\text{sup}} (2\| J(0)\|^2 + 2 \|\dot{J}(s) \|^2) \ud s \ud x\ud t \,,
\end{equation}
where $\delta = \sup_{(x,t) \in M\times[t_0,t_1]} \| \partial_s c(t,0,x) \|$ and $K_{\text{sup}}$ is a bound on $\max(K(y),0)$ with $K(y)$ is the maximum of the sectional curvatures at $y \in \mathcal{C}(M)$ for $y$ in a bounded neighborhood of $\underset{t \in [t_0,t_1]}{\bigcup} g_t(M)$ which is compact. Then, there exists $\delta$ sufficiently small such that for every $(x,t) \in M \times [t_0,t_1]$,
\begin{equation}
\int_{t_0}^{t_1} \int_M \int_0^1 (1-s) \langle R(\partial_s c,J)J,\partial_sc \rangle \ud s\ud x \ud t  \leq  1 \leq \int_{t_0}^{t_1} \int_M \int_0^1  \| \dot{J} \|^2  \ud s \ud x \ud t\,.
\end{equation}
\par
Now we study the second case, that is when $\int_{t_0}^{t_1} \int_M \int_0^1\| \dot{J} \|^2  \ud s \ud x \ud t \leq 1$. 
Applying once again inequality \eqref{Eq:PoincareInequality}, we obtain, using the Cauchy-Schwarz inequality,
\begin{multline}
\int_{t_0}^{t_1}\int_M \int_0^1 (1-s) \langle R(\partial_s c,J)J,\partial_sc \rangle \ud s \ud x \ud t \leq  \varepsilon K_{\text{sup}} \int_{t_0}^{t_1}\int_M \int_0^1  \|\dot{J}\|^2  \|J\|^2  \ud s \ud x \ud t \\ \leq  \varepsilon K_{\text{sup}} \left( \int_{t_0}^{t_1}\int_M \int_0^1  \|\dot{J}\|^4 \ud s \ud x \ud t \right)^{1/2}  \left(\int_{t_0}^{t_1}\int_M \int_0^1  \| J\|^4 \ud s \ud x \ud t\right)^{1/2} \,.
\end{multline}
We now remark that for each $t,x$, the space of Jacobi fields is finite dimensional and consequently, norms are equivalent so that there exists a positive constant $m$ that depends on $t,x$ such that 
\begin{equation}
\left(\int_0^1  \|\dot{J}\|^4 \ud s  \right)^{1/2} \leq m \int_0^1  \|\dot{J}\|^2 \ud s 
\end{equation} 
and 
\begin{equation}
\left(\int_0^1  \|J\|^4 \ud s  \right)^{1/2} \leq m \int_0^1  \|J\|^2 \ud s \,.
\end{equation}
By compactness of $M\times [t_0,t_1]$, the constant $m$ can be chosen independently of $t,x$ and therefore, there exists a constant $m'$ such that 
\begin{multline}
\int_{t_0}^{t_1}\int_M \int_0^1 (1-s) \langle R(\partial_s c,J)J,\partial_sc \rangle \ud s \ud x \ud t \leq  \\ \varepsilon K_{\text{sup}} m' \left( \int_{t_0}^{t_1} \int_M \int_0^1  \|\dot{J}\|^2 \ud s \ud x \ud t \right) \left(\int_{t_0}^{t_1}\int_M \int_0^1  \| J\|^2 \ud s \ud x \ud t\right) \,.
\end{multline}
 Then, inequality \eqref{Eq:SimpleIneq} leads to 
 \begin{multline}
\int_{t_0}^{t_1}\int_M \int_0^1 (1-s) \langle R(\partial_s c,J)J,\partial_sc \rangle \ud s \ud x \ud t  \leq  \varepsilon K_{\text{sup}} C m'\left( \int_{t_0}^{t_1} \int_M \int_0^1  \|\dot{J}\|^2 \ud s \ud x \ud t \right)  \,,
\end{multline}
with $M = \left(\int_{t_0}^{t_1}\int_M 2\| J(0)\|^2 + 2\int_0^1 \|\dot{J}(s) \|^2 \ud s \ud x \ud t\right)$.

\par 
Let us recall that our goal is to prove the existence of $\varepsilon>0$ such that
\begin{equation}\label{Eq:IneqToProve00}
\int_{t_0}^{t_1} \int_M \int_0^1 (1-s) \| \dot{J} \|^2  \ud s \ud x\ud t \geq \varepsilon \int_{t_0}^{t_1} \int_M \int_0^1 \| \dot{J} \|^2 + (1-s)\langle R(\partial_s c,J)J,\partial_sc\rangle \ud s \ud x\ud t\,,
\end{equation} 
which, in the first case, reads
\begin{equation}\label{Eq:IneqToProve01}
\int_{t_0}^{t_1} \int_M \int_0^1 (1-s) \| \dot{J} \|^2  \ud s \ud x\ud t \geq 2\varepsilon  \int_{t_0}^{t_1} \int_M \int_0^1 \| \dot{J} \|^2  \ud s \ud x\ud t\,,
\end{equation} 
and in the second case
\begin{equation}\label{Eq:IneqToProve2}
\int_{t_0}^{t_1} \int_M \int_0^1 (1-s) \| \dot{J} \|^2  \ud s \ud x\ud t \geq \varepsilon(1+ K_{\text{sup}} C m') \int_{t_0}^{t_1} \int_M \int_0^1 \| \dot{J} \|^2 \ud s \ud x\ud t\,.
\end{equation} 
The existence of $\varepsilon$ follows from the fact that the space of Jacobi fields is finite dimensional and the fact $M\times [t_0,t_1]$ is compact. It thus proves the result in the general case.
\par 
When the cone $\mathcal{C}(M)$ has non-positive sectional curvature, $K_{\text{sup}} = 0$ therefore, we only have to prove the existence of $\varepsilon$ such that 
\begin{equation}\label{Eq:IneqNegCurved}
\int_{t_0}^{t_1} \int_M \int_0^1 (1-s) \| \dot{J} \|^2  \ud s \ud x\ud t \geq \varepsilon  \int_{t_0}^{t_1} \int_M \int_0^1 \| \dot{J} \|^2  \ud s \ud x\ud t\,,
\end{equation} 
which does not require an a priori bound on the neighborhood. 
\par 
When $M = S_d(1)$, $\mathcal{C}(M)$ is flat and $\delta_0 = \infty$ and Jacobi fields are constant and the constant $\varepsilon$ does not depend on the neighborhood and is equal to $1/2$ as in Brenier's proof.
\end{proof}

\section*{Acknowledgements}
We would like to thank Yann Brenier and Klas Modin for stimulating discussions and a reviewer for his valuable comments which improved significantly this article.

\bibliographystyle{plain}
\bibliography{articles,SecOrdLandBig,sum_of_kernels,references}   

\begin{thebibliography}{10}

\bibitem{AmbrosioFlowWeak}
Luigi Ambrosio.
\newblock {\em The Flow Associated to Weakly Differentiable Vector Fields:
  Recent Results and Open Problems}, pages 181--193.
\newblock Springer US, Boston, MA, 2011.

\bibitem{Arnold1966}
Vladimir Arnold.
\newblock Sur la g\'eom\'etrie diff\'erentielle des groupes de {L}ie de
  dimension infinie et ses applications \`a l'hydrodynamique des fluides
  parfaits.
\newblock {\em Ann. Inst. Fourier (Grenoble)}, 16(fasc. 1):319--361, 1966.

\bibitem{benamou2000computational}
J-D. Benamou and Y.~Brenier.
\newblock A computational fluid mechanics solution to the {M}onge-{K}antorovich
  mass transfer problem.
\newblock {\em Numerische Mathematik}, 84(3):375--393, 2000.

\bibitem{Brenier1993}
Y.~Brenier.
\newblock The dual least action problem for an ideal, incompressible fluid.
\newblock {\em Archive for Rational Mechanics and Analysis}, 122(4):323--351,
  1993.

\bibitem{Brenier1991}
Yann Brenier.
\newblock Polar factorization and monotone rearrangement of vector-valued
  functions.
\newblock {\em Comm. Pure Appl. Math.}, 44(4):375--417, 1991.

\bibitem{Brenier1999}
Yann Brenier.
\newblock Minimal geodesics on groups of volume-preserving maps and generalized
  solutions of the {E}uler equations.
\newblock {\em Comm. Pure Appl. Math.}, 52(4):411--452, 1999.

\bibitem{BRENIER200355}
Yann Brenier.
\newblock Topics on hydrodynamics and volume preserving maps.
\newblock {\em Handbook of Mathematical Fluid Dynamics}, 2:55 -- 86, 2003.

\bibitem{Brenier2013}
Yann Brenier.
\newblock Remarks on the minimizing geodesic problem in inviscid incompressible
  fluid mechanics.
\newblock {\em Calc. Var. Partial Differential Equations}, 47(1-2):55--64,
  2013.

\bibitem{BressanConstantin2007}
Alberto {Bressan} and Adrian {Constantin}.
\newblock {Global conservative solutions of the Camassa-Holm equation.}
\newblock {\em {Arch. Ration. Mech. Anal.}}, 183(2):215--239, 2007.

\bibitem{Bressan2005}
Alberto Bressan and Massimo Fonte.
\newblock An optimal transportation metric for solutions of the
  {C}amassa-{H}olm equation.
\newblock {\em Methods Appl. Anal.}, 12(2):191--219, 2005.

\bibitem{MetricGeometryBurago}
D.~Burago, Y.~Burago, and S.~Ivanov.
\newblock A course in metric geometry.
\newblock {\em American Mathematical Soc.}, 2001.

\bibitem{CH}
Roberto {Camassa} and Darryl~D. {Holm}.
\newblock {An integrable shallow water equation with peaked solitons.}
\newblock {\em {Phys. Rev. Lett.}}, 71(11):1661--1664, 1993.

\bibitem{GeneralizedOT2}
L.~{Chizat}, G.~{Peyr{\'e}}, B.~{Schmitzer}, and F.-X. {Vialard}.
\newblock {Unbalanced Optimal Transport: Geometry and Kantorovich Formulation}.
\newblock {\em ArXiv e-prints}, August 2015.

\bibitem{GeneralizedOT1}
L.~{Chizat}, B.~{Schmitzer}, G.~{Peyr{\'e}}, and F.-X. {Vialard}.
\newblock {An Interpolating Distance between Optimal Transport and Fisher-Rao}.
\newblock {\em Found. Comp. Math.}, 2016.

\bibitem{Constantin2003}
A.~Constantin and B.~Kolev.
\newblock Geodesic flow on the diffeomorphism group of the circle.
\newblock {\em Comment. Math. Helv.}, 78(4):787--804, 2003.

\bibitem{Constantin2001}
Adrian Constantin.
\newblock On the scattering problem for the {C}amassa-{H}olm equation.
\newblock {\em R. Soc. Lond. Proc. Ser. A Math. Phys. Eng. Sci.},
  457(2008):953--970, 2001.

\bibitem{ConstantinEscher1998}
Adrian {Constantin} and Joachim {Escher}.
\newblock {Wave breaking for nonlinear nonlocal shallow water equations.}
\newblock {\em {Acta Math.}}, 181(2):229--243, 1998.

\bibitem{Constantin2008}
Adrian Constantin and David Lannes.
\newblock The hydrodynamical relevance of the camassa--holm and
  degasperis--procesi equations.
\newblock {\em Archive for Rational Mechanics and Analysis}, 192(1):165--186,
  2008.

\bibitem{danchin2001}
Rapha{\"e}l Danchin.
\newblock A few remarks on the {C}amassa-{H}olm equation.
\newblock {\em Differential Integral Equations}, 14(8):953--988, 2001.

\bibitem{Ebin1970}
David~G. Ebin and Jerrold Marsden.
\newblock Groups of diffeomorphisms and the motion of an incompressible fluid.
\newblock {\em Ann. of Math. (2)}, 92:102--163, 1970.

\bibitem{2017Elgindi}
T.~M. {Elgindi} and I.-J. {Jeong}.
\newblock {Finite-time Singularity Formation for Strong Solutions to the
  Boussinesq System}.
\newblock {\em ArXiv e-prints}, August 2017.

\bibitem{Escher12}
J.~{Escher} and B.~{Kolev}.
\newblock {Right-invariant Sobolev metrics of fractional order on the
  diffeomorphism group of the circle}.
\newblock {\em Journal of Geometric Mechanics}, 6(3):335 -- 372, September
  2014.

\bibitem{Escher2011}
Joachim Escher and Boris Kolev.
\newblock The degasperis--procesi equation as a non-metric euler equation.
\newblock {\em Mathematische Zeitschrift}, 269(3):1137--1153, 2011.

\bibitem{freed1989}
D.~S. Freed and D.~Groisser.
\newblock The basic geometry of the manifold of riemannian metrics and of its
  quotient by the diffeomorphism group.
\newblock {\em Michigan Math. J.}, 36(3):323--344, 1989.

\bibitem{Gallot1979}
S.~Gallot.
\newblock {\'E}quations diff{\'e}rentielles caract{\'e}ristiques de la
  sph{\`e}re.
\newblock {\em Annales scientifiques de l'Ecole Normale Superieure},
  12(2):235--267, 1979.

\bibitem{GallotHulinLafontaine}
S.~Gallot, D.~Hulin, and J.~Lafontaine.
\newblock {\em Riemannian Geometry}.
\newblock Universitext. Springer, 2004.

\bibitem{Merigot2016}
T.~O. Gallou{\"e}t and Q.~{M{\'e}rigot}.
\newblock A {L}agrangian scheme for the incompressible {E}uler equation using
  optimal transport.
\newblock ArXiv e-prints, May 2016.

\bibitem{JKOKFR}
T.~O. Gallou{\"e}t and L.~Monsaingeon.
\newblock A {J}{K}{O} splitting scheme for {K}antorovich-{F}isher-{R}ao
  gradient flows, 2016.

\bibitem{FGB2013}
F.~{Gay-Balmaz}, C.~{Tronci}, and C.~{Vizman}.
\newblock {Geometric dynamics on the automorphism group of principal bundles:
  geodesic flows, dual pairs and chromomorphism groups}.
\newblock {\em Journal of Geometric Mechanics}, 5:39--84, 2013.

\bibitem{GRUNERT2011}
Katrin Grunert, Helge Holden, and Xavier Raynaud.
\newblock Lipschitz metric for the periodic camassa--holm equation.
\newblock {\em Journal of Differential Equations}, 250(3):1460 -- 1492, 2011.

\bibitem{Holm1998}
D.~D. Holm, J.~E. Marsden, and T.~S. Ratiu.
\newblock The {Euler-Poincar\'{e}} equations and semidirect products with
  applications to continuum theories.
\newblock {\em Adv. Math.}, 137:1--81, 1998.

\bibitem{KhesinCurvature}
B.~{Khesin}, J.~{Lenells}, G.~{Misiolek}, and S.~C. {Preston}.
\newblock {Curvatures of Sobolev metrics on diffeomorphism groups}.
\newblock {\em Pure and Applied Mathematics Quarterly}, 9(2):291 -- 332, 2013.

\bibitem{Khesin2013}
B.~Khesin, J.~Lenells, G.~Misio{\l}ek, and S.~C. Preston.
\newblock Geometry of {D}iffeomorphism {G}roups, {C}omplete integrability and
  {G}eometric statistics.
\newblock {\em Geom. Funct. Anal.}, 23(1):334--366, 2013.

\bibitem{khesin2008geometry}
B.~Khesin and R.~Wendt.
\newblock {\em The geometry of infinite-dimensional groups}, volume~51.
\newblock Springer Science \&amp; Business Media, 2008.

\bibitem{Michor1993}
I.~Kol{\'a}{\v{r}}, P.~W. Michor, and J.~Slov{\'a}k.
\newblock {\em Natural operations in differential geometry}.
\newblock Springer-Verlag, Berlin, 1993.

\bibitem{new2015kondratyev}
S.~Kondratyev, L.~Monsaingeon, and D.~Vorotnikov.
\newblock A new optimal trasnport distance on the space of finite {R}adon
  measures.
\newblock {\em Adv. Differential Equations}, 21(11):1117--1164, 2016.

\bibitem{Kouranbaeva1999}
Shinar Kouranbaeva.
\newblock The {C}amassa-{H}olm equation as a geodesic flow on the
  diffeomorphism group.
\newblock {\em J. Math. Phys.}, 40(2):857--868, 1999.

\bibitem{Lenells2005}
Jonatan {Lenells}.
\newblock {Conservation laws of the Camassa-Holm equation.}
\newblock {\em {J. Phys. A, Math. Gen.}}, 38(4):869--880, 2005.

\bibitem{LieroMielkeSavareLong}
M.~{Liero}, A.~{Mielke}, and G.~{Savar{\'e}}.
\newblock {Optimal Entropy-Transport problems and a new Hellinger-Kantorovich
  distance between positive measures}.
\newblock {\em ArXiv e-prints}, August 2015.

\bibitem{LieroMielkeSavareShort}
M.~{Liero}, A.~{Mielke}, and G.~{Savar{\'e}}.
\newblock {Optimal transport in competition with reaction: the
  Hellinger-Kantorovich distance and geodesic curves}.
\newblock {\em SIAM J. Image Analysis}, 48(4):2869--2911, 2016.

\bibitem{lott2008some}
J.~Lott.
\newblock Some geometric calculations on {W}asserstein space.
\newblock {\em Communications in Mathematical Physics}, 277(2):423--437, 2008.

\bibitem{LuoShvydkoy}
Xue Luo and Roman Shvydkoy.
\newblock 2d homogeneous solutions to the euler equation.
\newblock {\em Communications in Partial Differential Equations},
  40(9):1666--1687, 2015.

\bibitem{OTmaasrumpf}
J.~Maas, M.~Rumpf, C.~Sch{\"o}nlieb, and S.~Simon.
\newblock A generalized model for optimal transport of images including
  dissipation and density modulation.
\newblock {\em ESAIM: Mathematical Modelling and Numerical Analysis}, 49(6),
  Apr 2015.
\newblock arXiv:1504.01988.

\bibitem{McKean2004}
Henry~P. McKean.
\newblock Breakdown of the {C}amassa-{H}olm equation.
\newblock {\em Comm. Pure Appl. Math.}, 57(3):416--418, 2004.

\bibitem{Merigot2015}
Q.~{M{\'e}rigot} and J.-M. {Mirebeau}.
\newblock {Minimal geodesics along volume preserving maps, through
  semi-discrete optimal transport}.
\newblock {\em ArXiv e-prints}, May 2015.

\bibitem{Michor2008b}
P.~W. Michor.
\newblock {\em Topics in Differential Geometry}, volume~93 of {\em Graduate
  Studies in Mathematics}.
\newblock American Mathematical Society, Providence, RI, 2008.

\bibitem{Michor2005}
Peter~W. Michor and David Mumford.
\newblock Vanishing geodesic distance on spaces of submanifolds and
  diffeomorphisms.
\newblock {\em Doc. Math.}, 10:217--245, 2005.

\bibitem{Misiolek2002}
G.~Misiolek.
\newblock Classical solutions of the periodic {C}amassa-{H}olm equation.
\newblock {\em Geometric {\&} Functional Analysis GAFA}, 12(5):1080--1104,
  2002.

\bibitem{MisolekPreston2010}
Gerard Misiolek and Stephen~C. Preston.
\newblock Fredholm properties of riemannian exponential maps on diffeomorphism
  groups.
\newblock {\em Inventiones mathematicae}, 179(1):191--227, 2010.

\bibitem{Modin2012}
K.~{Modin}.
\newblock {Generalised Hunter-Saxton equations, optimal information transport,
  and factorisation of diffeomorphisms}.
\newblock {\em Journal of Geometric Analysis}, 25(2):1306--1334, April 2015.

\bibitem{MoserLemma}
J.~Moser.
\newblock On the volume elements on a manifold.
\newblock {\em Trans. Amer. Math. Soc.}, 120:286--294, 1965.

\bibitem{MumfordMichorHdiv}
D.~{Mumford} and P.~W. {Michor}.
\newblock {On Euler's equation and `EPDiff'}.
\newblock {\em Journal of Geometric Mechanics}, 5:319 -- 344, 2013.

\bibitem{OttoPorousMedium}
Felix Otto.
\newblock The geometry of dissipative evolution equations: The porous medium
  equation.
\newblock {\em Communications in Partial Differential Equations},
  26(1-2):101--174, 2001.

\bibitem{Rezankhanlou2015}
F~Rezakhanlou.
\newblock Optimal transport problems for contact structures, 2015.

\bibitem{Metamorphosis2005}
A.~Trouv\'e and L.~Younes.
\newblock Metamorphoses through lie group action.
\newblock {\em Foundations of Computational Mathematics}, 5(2):173--198, 2005.

\bibitem{villani2008optimal}
C.~Villani.
\newblock {\em Optimal transport: old and new}, volume 338.
\newblock Springer Science \&amp; Business Media, 2008.

\end{thebibliography}

\end{document}